\documentclass{amsart}

\usepackage{amsmath,amssymb,amsbsy,amsthm,amsfonts}

\usepackage{xcolor}

\begin{document}

\newcommand{\ddd}{\,{\rm d}}
%\numberwithin{equation}{section} \marginparwidth=2cm

\def\note#1{\marginpar{\small #1}}

\def\tens#1{\pmb{\mathsf{#1}}}

\def\vec#1{\boldsymbol{#1}}

\def\norm#1{\left|\!\left| #1 \right|\!\right|}

\def\fnorm#1{|\!| #1 |\!|}

\def\abs#1{\left| #1 \right|}

\def\ti{\text{I}}

\def\tii{\text{I\!I}}

\def\tiii{\text{I\!I\!I}}

\newcommand{\supp}{\operatorname{supp}}
\newcommand{\esup}{\operatorname{ess\,sup\,}\displaylimits}

\def\diver{\mathop{\mathrm{div}}\nolimits}

\def\grad{\mathop{\mathrm{grad}}\nolimits}

\def\Div{\mathop{\mathrm{Div}}\nolimits}

\def\Grad{\mathop{\mathrm{Grad}}\nolimits}

\def\tr{\mathop{\mathrm{tr}}\nolimits}

\def\cof{\mathop{\mathrm{cof}}\nolimits}

\def\det{\mathop{\mathrm{det}}\nolimits}

\def\lin{\mathop{\mathrm{span}}\nolimits}

\def\pr{\noindent \textbf{Proof: }}

\def\pp#1#2{\frac{\partial #1}{\partial #2}}

\def\dd#1#2{\frac{\d #1}{\d #2}}

\def\bA{\tens{A}}
\def\bH{\tens{H}}

\def\T{\mathcal{T}}

\def\R{\mathcal{R}}

\def\bx{\vec{x}}

\def\be{\vec{e}}

\def\bef{\vec{f}}

\def\bec{\vec{c}}

\def\bs{\vec{s}}

\def\ba{\vec{a}}

\def\bn{\vec{n}}

\def\bphi{\vec{\varphi}}

\def\btau{\vec{\tau}}

\def\bc{\vec{c}}

\def\bg{\vec{g}}

\def\mO{\mathcal{O}}
\def\pmO{\partial\mathcal{O}}

\def\bE{\tens{\varepsilon}}
\def\bsig{\tens{\sigma}}
\def\bbeta{\tens{\beta}}

\def\bW{\tens{W}}

\def\bT{\tens{T}}

\def\bxi{\tens{\xi}}

\def\bD{\tens{D}}

\def\bF{\tens{F}}

\def\bB{\tens{B}}

\def\bV{\tens{V}}

\def\bS{\tens{S}}

\def\bI{\tens{I}}

\def\bi{\vec{i}}

\def\bv{\vec{v}}

\def\bfi{\vec{\varphi}}

\def\bk{\vec{k}}

\def\b0{\vec{0}}

\def\bom{\vec{\omega}}

\def\bw{\vec{w}}

\def\p{\pi}

\def\bu{\vec{u}}
\def\bz{\vec{z}}
\def\bep{\vec{e}_{\textrm{p}}}
\def\dbep{\dot{\vec{e}}_{\textrm{p}}}
\def\bee{\vec{e}_{\textrm{el}}}
\def\dbee{\dot{\vec{e}}_{\textrm{el}}}

\def\ID{\mathcal{I}_{\bD}}

\def\IP{\mathcal{I}_{p}}

\def\Pn{(\mathcal{P})}

\def\Pe{(\mathcal{P}^{\eta})}

\def\Pee{(\mathcal{P}^{\varepsilon, \eta})}

\def\dx{\,{\rm dx}}

%------------------------------------------------

\newtheorem{Theorem}{Theorem}[section]

\newtheorem{Example}{Example}[section]

\newtheorem{Lemma}{Lemma}[section]

\newtheorem{Rem}{Remark}[section]

\newtheorem{Def}{Definition}[section]

\newtheorem{Col}{Corollary}[section]

\numberwithin{equation}{section}

\title[Regularity results in elasto-plasticity theory with hardening]{Regularity results for two standard models in elasto-perfect-plasticity theory with hardening}

\author[M.~Bul\'{i}\v{c}ek]{Miroslav Bul\'i\v{c}ek}
\address{Mathematical Institute of Charles University\\
Sokolovsk\'{a} 83, 186 75 Prague, Czech Republic}
\email{mbul8060@karlin.mff.cuni.cz}
\thanks{M.~Bul\'{\i}\v{c}ek's work is supported by the project 20-11027X financed by GA\v{C}R. Authors are also thankful to the Hausdorff center in Bonn. M.~Bul\'{\i}\v{c}ek is a  member of the Ne\v{c}as Center for Mathematical Modeling}

\author[J. Frehse]{Jens Frehse}
\address{Institute of Applied Mathematics, University of Bonn\\
Endenicher Allee 60, D-53121 Bonn, Germany}
\email{aaa@iam.uni-bonn.de}
%\thanks{Jens Frehse acknowledges Jind\v{r}ich Ne\v{c}as Center for Mathematical
%Modeling, the project LC06052, for its support.}

\author[M. Specovius-Neugebauer]{Maria Specovius--Neugebauer }
\address{Fachbereich 17, Mathematik, Universit\"{a}t Kassel\\
HeinrichPlettStr.40, D-34132 Kassel, Germany}
\email{specovi@mathematik.uni-kassel.de}
%\thanks{Josef M\'{a}lek's contribution is a part of the research
%project MSM 0021620839 financed by M\v{S}MT; the support of GA\v{C}R
%201/09/0917 is also acknowledged.}

%

\dedicatory{Dedicated to Umberto Mosco}

\begin{abstract}
We consider two most studied standard models in the theory of elasto-plasticity with hardening in arbitrary dimension $d\ge 2$, namely,  the kinematic hardening and the isotropic hardening problem. While the existence and uniqueness of the solution is very well known, the optimal regularity up to the boundary remains an open problem. Here, we show that in the interior we have Sobolev regularity for the stress and hardening while for their time derivatives we have the ``half" derivative with the spatial and time variable. This was well known for the limiting problem but we show that these estimates are uniform and independent of the order of approximation. The main novelty consist of estimates near the boundary. We show that for the stress and the hardening parameter, we control tangential derivative in the Lebesgue space~$L^2$, and for time derivative of the stress and the hardening we control the ``half" time derivative and also spatial tangential derivative. Last, for the normal derivative, we show that the stress and the hardening have the $3/5$ derivative with respect to the normal and for the time derivative of the stress and the hardening we show they have the $1/5$ derivative with respect to the normal direction, provided we consider the kinematic hardening or near the Dirichlet boundary. These estimates are independent of dimension. In case, we consider the isotropic hardening near the Neumann boundary we shall obtain $W^{\alpha,2}$ regularity for the stress and the hardening with some $\alpha>1/2$ depending on the dimension and $W^{\beta,2}$ with some $\beta > 1/6$ for the time derivative of the stress and the hardening. Finally, in case of kinematic hardening the same regularity estimate holds true also for the velocity gradient.
\end{abstract}

\maketitle

\section{Introduction}
In this paper we deal with the regularity estimates for solutions to some  models of linearized  elasto--plasticity with hardening. We have mainly two cases in mind, the isotropic hardening and the kinematic harding. Our main goal is to provide the uniform estimates on the Cauchy stress and the hardening parameter and their time derivatives in fractional Sobolev spaces up to the boundary and consequently by using an interpolation technique also to improve the available regularity results for the small strain tensor.

\subsection{Physical background}

To describe the problem in more details, we shall assume that a body occupies a Lipschitz set $\mathcal{O} \subset \mathbb{R}^d$ and we a~priori assume that considered deformations are small. Therefore, the initial, the current and the preferred (natural) configurations coincide and we can approximate the strain tensor by the linearized strain tensor $\bE(\bu)$, which is defined as
\begin{equation}\label{linstrain}
\bE(\bu):= \frac12 (\nabla \bu + (\nabla \bu)^T)\,,
\end{equation}
where $\bu:(0,T)\times \mathcal{O} \to \mathbb{R}^d$ is the displacement and the interval $(0,T)$ represents the loading parameter, which we call ``time" in what follows. We also assume that the density is constant and that the inertial effects can be neglected. Then the balance of linear momentum for the quasi-static deformation takes the form
\begin{equation}
   -\diver \bsig = \bef \quad \textrm{ in } [0,T]\times \mathcal{O}, \label{Newton}
\end{equation}
where $\bsig:(0,T)\times \mathcal{O}\to \mathbb{R}^{d\times d}_{sym}$ is the Cauchy stress and $\bef:(0,T)\times \mathcal{O}\to \mathbb{R}^d$ denotes the density of given external body forces. To complete the problem \eqref{linstrain}--\eqref{Newton} it remains to prescribe the boundary and initial conditions, which we shall do later, and also to characterize the relationship between $\bsig$ and $\bE(\bu)$. Since we deal with elasto--platic effects, we assume that the linearized strain $\bE(\bu)$ can be decomposed into the elastic part $\bee$ and the plastic part $\bep$, i.e.,
\begin{equation}
     \bE(\bu) = \bee + \bep\, \label{Sum}
\end{equation}
and that the elastic response of the material is given by the Helmholtz potential $\psi^*:\mathbb{R}^{d\times d}_{sym} \to \mathbb{R}$, which is supposed to be a strictly convex function vanishing at zero and exploding at infinity and the elastic strain is related to the stress through
\begin{equation}
    \bsig = \frac{\partial \psi^*(\bee)}{\partial \bee} \quad \Leftrightarrow \quad  \bee = \frac{\partial \psi(\bsig)}{\partial \bsig}\,, \label{isotropic_1}
\end{equation}
where $\psi$ is the conjugate function to $\psi^*$ defined as
$$
\psi(\bsig):=\sup_{\bee} \left( \bsig \cdot \bee - \psi^*(\bee)\right)\,.
$$
Furthermore, we require in the paper that there exists a constant fourth order tensor $\bA \in \mathbb{R}^{d\times d}_{sym}\times \mathbb{R}^{d\times d}_{sym}$ such that for all $\bsig \in \mathbb{R}^{d\times d}_{sym}$
\begin{equation}\label{ellip2}
\bA\equiv \frac{\partial^2 \psi(\bsig)}{\partial \bsig \partial \bsig}.
\end{equation}
Then, evidently, the relation \eqref{isotropic_1} can be rewritten as
\begin{equation}
    \bsig = \bA^{-1}\bee \quad \Leftrightarrow \quad  \bee = \bA \bsig\,. \label{isotropic_1_1}
\end{equation}
Note that thanks to \eqref{ellip2}, the tensor $\bA$ is symmetric. In addition, since $\psi$ is assumed to be convex, we certainly know that $\bA$ is invertible and therefore \eqref{isotropic_1_1} makes good sense.

Concerning the plastic strain, we  assume that (usually $\bep$ is considered to be relevant to incompressible motion)
\begin{equation}
   \tr \bep = 0\,, \label{incompress}
\end{equation}
where $\tr$ denotes the trace of $\bep$. 
Further, we need to specify under which conditions it may appear and how is related to the ``hardening". In the paper, we shall assume two cases, the {\bf isotropic hardening}, which is described by a scalar function
$$\xi:[0,T]\times \mO \to \mathbb{R},$$
which is related to the plastic strain by the following flow rule (here $H$ is a given positive constant)
\begin{equation}\label{hard_iso}
H\dot{\xi} = |\dot{\bep}|.
\end{equation}
Here, and also in what follows, we use the ``dot" to abbreviate the partial derivative with respect to the variable $t$, i.e., 
$$
\dot{z}:=\frac{\partial z}{\partial t}
$$
for any function $z$ depending on $t$.
Furthermore, we require the von~Mises condition
$$
|\bsig_D| - \xi \le \kappa,
$$
and we assume that there is no plastic strain if the above inequality is strict, i.e.,
$$
|\bsig_D| - \xi < \kappa \implies \dbep=0.
$$
On the other hand, if $|\bsig_D| - \xi = \kappa$, we require that
$$
\frac{\bsig_D}{|\bsig_D|} = \frac{\dbep}{|\dbep|}.
$$
To summarize, we can write these conditions in a more compact form
\begin{equation}
   \begin{split}
  \dbep = \dot{\lambda} \frac{\bsig_D}{|\bsig_D|} \quad &\textrm{ with } \quad \dot{\lambda} \ge 0\,, \quad |\bsig_D| - \xi \le \kappa \quad \textrm{ and } \quad \dot{\lambda} \left( |\bsig_D| -\xi- \kappa \right) = 0\,.
  \end{split}\label{KT_iso}
\end{equation}

In a similar way, we shall define the {\bf kinematic hardening}. Hence, we assume that the hardening parameter
$$
\bxi:[0,T]\times \mO \to \mathbb{R}^{d\times d}_{sym}
$$
obeys the following flow rule (here $\bH$ is the symmetric positively definite fourth order tensor)
\begin{equation}\label{hard_kin}
\bH\dot{\bxi} = \dbep.
\end{equation}
The related von~Mises condition then reads as
$$
|\bsig_D - \bxi_D|  \le \kappa,
$$
and we assume that there is no plastic strain if the above inequality is strict, i.e.,
$$
|\bsig_D - \bxi_D| < \kappa \implies \dbep=0.
$$
On the other hand, if $|\bsig_D - \bxi_D| = \kappa$, we require that
$$
\frac{\bsig_D-\bxi_D}{|\bsig_D-\bxi_D|} = \frac{\dbep}{|\dbep|}.
$$
Again, the above conditions can be equivalently rewritten as
\begin{equation}
   \begin{split}
  \dbep = \dot{\lambda} \frac{\bsig_D-\bxi_D}{|\bsig_D-\bxi_D|}  &\textrm{ with }  \dot{\lambda} \ge 0\,, \quad |\bsig_D-\bxi_D|  \le \kappa \quad \textrm{ and } \quad \dot{\lambda} \left( |\bsig_D-\bxi_D|- \kappa \right) = 0\,.
  \end{split}\label{KT_kin}
\end{equation}

We do not claim that two models introduced above are the only ones of physical and engineering interest. Indeed, there is a lot of models describing elastic and plastic deformation with memory effects, see for example the book \cite{Al98}, and each of them can be used in a specific situation. The models of kinematic and isotopic hardening are just two most prototypic examples. On the other hand, it seems that the models discussed in this paper usually have better regularity properties than the others, and therefore it was also our motivation to focus on them.

\subsection{Weak formulation of the problem and the main result}
To complete the problem, we have to specify the boundary and initial values. We start with the hardening variable and we assume that $\bxi(0,x)\equiv 0$ in case of kinematic hardening and that $\xi(0,x)=0$ in case of isotropic hardening. Next, we assume that the boundary of the domain can be decomposed onto two smooth sets $\partial \mO_D$ - the Dirichlet part,  and $\partial\mO_N$ - the Neumann part with the unit normal outward vector denoted by $\bn$. Finally, we prescribe the data $\bu_0 :[0,T]\times \mO \to \mathbb{R}^d$ and $\bsig_0:[0,T] \times \mO \to \mathbb{R}^{d\times d}_{sym}$,  and we require that $\bu = \bu_0$ on $[0,T]\times \partial \mO_D$ (Dirichlet data for the displacement - the displacement on the part $\partial \mO_D$) and that $\bu(0)=\bu_0(0)$ (the initial displacement), and that $\bsig \cdot \bn = \bsig_0 \cdot \bn$ on $[0,T]\times \mO_N$ (the Neumann data - the traction on $\partial \mO_N$) and $\bsig(0)=\bsig_0(0)$ (the initial value of the Cauchy stress). Note here, that there is a necessary compatibility condition $\bE(\bu_0(0)) = \bA \bsig_0(0)$. To recall all data, we also have a given body forces $\bef:[0,T]\times \mO \to \mathbb{R}^d$ and the material parameters\footnote{For simplicity, we assume that the material parameters are constant. Nevertheless, we could allow them to be time dependent.}  $\kappa\ge C_1>0$ for some positive constant $C_1$, positively definite symmetric fourth order tensor $\bA$ and the constant $H>0$ in case of isotropic harding and the constant fourth order symmetric positively definite tensor $\bH$ in case of kinematic hardening. Here, positively definite means that there exists a positive constant $C_1$ such that $C_1\le H$ and such that for all $\btau \in \mathbb{R}^{d\times d}_{sym}$ there holds
\begin{equation}\label{ellip}
C_1 |\btau|^2\le  \bA \btau \cdot \btau \le C_1^{-1} |\btau|^2\quad \textrm{ and } \quad C_1 |\btau|^2\le \bH\btau\cdot \btau\le C_1^{-1} |\btau|^2.
\end{equation}
Finally, we can summarize the above description and formulate the problem of elasto-plastic hardening in the following classical way:

\bigskip

\paragraph{\bf Kinematic hardening:}  We look for a quintuple $(\bsig, \bxi, \bu, \bee, \bep ):[0,T]\times \mathcal{O} \to \mathbb{R}^{d\times d}_{sym}\times \mathbb{R}^{d\times d}_{sym} \times \mathbb{R}^d\times \mathbb{R}^{d\times d}_{sym}\times \mathbb{R}^{d\times d}_{sym}$ such that
\begin{equation}
\begin{aligned}
-\diver \bsig &= \bef, \quad \bH\dot{\bxi} = \dbep,\quad \bE(\bu)=\bee+\bep, \quad \bee = \bA \bsig &&\textrm{ in } [0,T]\times \mathcal{O}\,,\\
\dbep &= |\dbep| \frac{\bsig_D-\bxi_D}{\kappa} &&\textrm{ in } [0,T]\times \mathcal{O}\,,\\
\kappa& \ge |\bsig_D -\bxi_D|\, \textrm{ and  } \, |\dbep|(|\bsig_D-\bxi_D| - \kappa) = 0 &&\textrm{ in } [0,T]\times \mathcal{O}\,,\\
\bu&=\bu_0 &&\textrm{ on } [0,T]\times \partial \mathcal{O}_D\,,\\
%\bu\cdot \bn &=\bu_0\cdot \bn, \quad (\bsig \bn)_{\btau} = (\bef_{\! \bn})_{\btau}  &&\textrm{ on } [0,T]\times \partial \mathcal{O}_M\,,\\
\bsig \bn &= \bsig_0 \bn  &&\textrm{ on } [0,T]\times \partial \mathcal{O}_N\,, \\
\bsig(0)&=\bsig_0, \quad \bxi(0)=0, \quad \bu(0)=\bu_0(0) &&\textrm{ in } \mathcal{O}\,.
\end{aligned}
\label{PR11}
\end{equation}
where $T>0$ is the given length of the time\footnote{In fact, we should not call it time interval, since $t$ corresponds to the loading parameter.} interval, the given threshold $\kappa>0$ is  a von~Mises condition,  $\bef:[0,T]\times \mathcal{O} \to \mathbb{R}^d$ are the given volume forces,  the stress $\bsig_0:[0,T]\times \mathcal{O}\to \mathbb{R}^{d\times d}_{sym}$ represents the initial value $\bsig(0)$ and the traction $\bsig \bn$ and the prescribed displacement on the boundary $[0,T]\times \partial \mathcal{O}_D$ and the initial displacement is represented by $\bu_0:[0,T]\times \mathcal{O}\to \mathbb{R}^d$. Here the symbol $\bn$ denotes the outer normal vector on $\pmO$.

\bigskip

\paragraph{\bf Isotropic hardening:}
 We look for a quintuple $(\bsig, \xi, \bu, \bee, \bep ):[0,T]\times \mathcal{O} \to \mathbb{R}^{d\times d}_{sym}\times \mathbb{R} \times \mathbb{R}^d\times \mathbb{R}^{d\times d}_{sym}\times \mathbb{R}^{d\times d}_{sym}$ such that
\begin{equation}
\begin{aligned}
-\diver \bsig &= \bef, \quad H\dot{\xi} = |\dbep|,\quad \bE(\bu)=\bee+\bep, \quad \bee = \bA \bsig &&\textrm{ in } [0,T]\times \mathcal{O}\,,\\
\dbep &= |\dbep| \frac{\bsig_D}{\kappa+\xi} &&\textrm{ in } [0,T]\times \mathcal{O}\,,\\
\kappa& \ge |\bsig_D|-\xi\, \textrm{ and  } \, |\dbep|(|\bsig_D| - \kappa -\xi ) = 0 &&\textrm{ in } [0,T]\times \mathcal{O}\,,\\
\bu&=\bu_0 &&\textrm{ on } [0,T]\times \partial \mathcal{O}_D\,,\\
%\bu\cdot \bn &=\bu_0\cdot \bn, \quad (\bsig \bn)_{\btau} = (\bef_{\! \bn})_{\btau}  &&\textrm{ on } [0,T]\times \partial \mathcal{O}_M\,,\\
\bsig \bn &= \bsig_0 \bn  &&\textrm{ on } [0,T]\times \partial \mathcal{O}_N\,, \\
\bsig(0)&=\bsig_0, \quad \xi(0)=0, \quad \bu(0)=\bu_0(0) &&\textrm{ in } \mathcal{O}\,.
\end{aligned}
\label{He11}
\end{equation}

Furthermore, to simplify the notation, we require that (indeed, it is just simplification of a notation, in fact, to be able to solve the problem, the existence of $\bsig_0$ fulfilling this equation is necessary)
$$
\diver \bsig_0 = \bef \quad \textrm{almost everywhere in } (0,T)\times \mO.
$$
Then,  we choose a proper subspace of the Sobolev space $W^{1,2}(\mathcal{O}; \mathbb{R}^d)$, which will be used in what follows
$$
\mathcal{V}:=\{\bv \in W^{1,2}(\mathcal{O};\mathbb{R}^d); \; \bv = \b0 \textrm{ on } \partial \mathcal{O}_D\}
$$
and we define the set of admissible stresses as
\begin{align*}
\mathcal{F}_k(t)&:=\left\{(\bsig,\bxi)\in   L^2(\Omega; \mathbb{R}^{d\times d}_{sym})\times  L^2(\Omega; \mathbb{R}^{d\times d}_{sym}); \;
|\bsig_D-\bxi_D|\le \kappa,\right.\\
&\qquad \left.\; \, \textrm{ and for all } \bv \in \mathcal{V} \textrm{ there holds }\int_{\mathcal{O}} (\bsig-\bsig_0) \cdot \bE(\bv) \ddd x =0 \right\}\\
\intertext{and}
\mathcal{F}_i(t)&:=\left\{(\bsig,\xi) \in L^2(\Omega; \mathbb{R}^{d\times d}_{sym})\times L^2(\Omega;\mathbb{R}); \;
|\bsig_D|\le \kappa+\xi,\right.\\
&\qquad \left.\; \textrm{ and for all } \bv \in \mathcal{V} \textrm{ there holds } \int_{\mathcal{O}} (\bsig -\bsig_0)\cdot \bE(\bv) \ddd x = 0 \right\}.
\end{align*}
Notice here, that $\mathcal{F}_i$ corresponds to isotropic hardening while the set $\mathcal{F}_k$ is related to kinematic hardening and we can introduce the following definitions.
\begin{Def}[Kinematic hardening]\label{D1}
Let $\mO\subset \mathbb{R}^d$ be a Lipschitz domain. Assume that $\bsig_0\in W^{1,2}(0,T; L^2(\Omega;\mathbb{R}^{d\times d}_{sym}))$ and $\bu_0 \in W^{1,2}(0,T; W^{1,2}(\mO;\mathbb{R}^d))$. We say that $(\bsig, \bxi) \in W^{1,2}(0,T; L^2(\mO;\mathbb{R}_{sym}^{d\times d}))\times W^{1,2}(0,T; L^2(\mO;\mathbb{R}_{sym}^{d\times d}))$ is a weak solution to \eqref{PR11} if $\bsig(0)=\bsig_0$, $\bxi(0)=0$ and for almost all $t\in (0,T)$ there holds $(\bsig(t), \bxi(t))\in \mathcal{F}_k(t)$ and, in addition, we require that  for almost all $t\in (0,T)$ and all $(\tilde{\bsig}, \tilde{\bxi}) \in \mathcal{F}_k(t)$ there holds
\begin{equation}\label{wf1}
\int_{\mO} \bA \dot\bsig(t) \cdot (\bsig(t)-\tilde{\bsig})) + \bH \dot{\bxi} \cdot (\bxi-\tilde{\bxi})\ddd x \le \int_{\mO} \bE(\dot\bu_0) \cdot (\bsig(t)-\tilde{\bsig})\ddd x.
\end{equation}
\end{Def}
In a very similar way, we can also introduce the notion of a weak solution to the isotropic model \eqref{He11}, where we shall replace $\mathcal{F}_k(t)$ by $\mathcal{F}_i(t)$ in a natural way.
\begin{Def}[Isotropic hardening]\label{D2}
Let $\mO\subset \mathbb{R}^d$ be a Lipschitz domain. Assume that $\bsig_0\in W^{1,2}(0,T; L^2(\Omega;\mathbb{R}^{d\times d}_{sym}))$ and $\bu_0 \in W^{1,2}(0,T; W^{1,2}(\mO;\mathbb{R}^d))$. We say that $(\bsig, \xi) \in W^{1,2}(0,T; L^2(\mO;\mathbb{R}_{sym}^{d\times d}))\times W^{1,2}(0,T; L^2(\mO;\mathbb{R})$ is a weak solution to \eqref{He11} if $\bsig(0)=\bsig_0$, $\xi(0)=0$ and for almost all $t\in (0,T)$ there holds $(\bsig(t), \xi(t))\in \mathcal{F}_i(t)$ and, in addition, we require that  for almost all $t\in (0,T)$ and all $(\tilde{\bsig}, \tilde{\xi}) \in \mathcal{F}_i(t)$ there holds
\begin{equation}\label{wf2}
\int_{\mO} \bA \dot\bsig(t) \cdot (\bsig(t)-\tilde{\bsig})) + H \dot{\xi} (\xi-\tilde{\xi})\ddd x \le \int_{\mO} \bE(\dot\bu_0) \cdot (\bsig(t)-\tilde{\bsig})\ddd x.
\end{equation}
\end{Def}
Before stating the main result of the paper, we introduce the safety load condition for the initial data $\bsig_0(0)$, namely
\begin{equation}
\|\bsig_{0D}(0)\|_{\infty}<\kappa.\label{SLMB}
\end{equation}
The existence of weak solution to kinematic or isotropic hardening problem in the sense of Definitions~\ref{D1}--\ref{D2} is very standard. However, to be able to talk also about the displacement and to obtain regularity results, one usually needs to assume certain compatibility condition on data. In the available literature, the authors usually consider more restrictive assumption than \eqref{SLMB}, which is however more related to the classical problems of elasto--plasticity without hardening. Nevertheless, in our setting, the assumption \eqref{SLMB} is sufficient, since it leads to the standard safety load condition. Indeed, for kinematic hardening we can set $\bxi_0(t):=\bsig_0(t)-\bsig_0(0)$ and then directly we also have
\begin{equation}
(\bsig_0(t),\bxi_0(t))\in \mathcal{F}_k(t), \quad \sup_{t\in (0,T)}\|\bsig_{D0}(t)-\bxi_{0D}(t)\|_{\infty}< \kappa, \label{SL}
\end{equation}
provided \eqref{SLMB} holds. Similarly, in the isotropic hardening case, we se $\xi_0(t):=|\bsig_{0D}(t)|-|\bsig_{0D}(0)|$ and we again have
\begin{equation}
(\bsig_0(t),\xi_0(t))\in \mathcal{F}_i(t), \quad \sup_{t\in (0,T)} \||\bsig_{0D}(t)|-\xi(t)\|_{\infty}<\kappa, \label{SLH}
\end{equation}
provided \eqref{SLMB} holds.

Finally, we state the main results of the paper. We consider an approximated problem and show not only the convergence to the original problem but also regularity estimates that are uniform with respect to the approximation parameter.  The first one is for the kinematic hardening model. For the approximation, we introduce a new class of admissible stresses as
$$
\begin{aligned}
\mathcal{F}_{el}(t)&:=\left\{(\bsig,\bxi) \in L^2(\Omega; \mathbb{R}^{d\times d}_{sym})\times L^2(\Omega; \mathbb{R}^{d\times d}_{sym}); \right.\\
&\qquad \left. \textrm{ and for all } \bv \in \mathcal{V} \textrm{ there holds } \; \int_{\mathcal{O}} (\bsig-\bsig_0) \cdot \bE(\bv) \ddd x = 0 \right\}.
\end{aligned}
$$
and our result for kinematic hardening reads as follows.
\begin{Theorem}[Kinematic hardening]\label{TPR}
Let all assumptions of Definition~\ref{D1} be satisfied. Then for all $\mu>0$ there exists a unique triple $(\bsig^{\mu},\bxi^{\mu},\bu^{\mu})$ such that $\bu^{\mu}-\bu_0\in W^{1,2}(0,T; W^{1,2}_0(\mO;\mathbb{R}^d))$, $\bxi^{\mu}(0)=0$, $\bsig^{\mu}=\bsig_0(0)$ and
\begin{equation}
\left. \begin{aligned}
&\bA \dot{\bsig}^{\mu} + \mu^{-1}(|\bsig_D^{\mu}-\bxi_D|-\kappa)_+ \frac{\bsig^{\mu}_D-\bxi_D}{|\bsig^{\mu}_D-\bxi_D|} = \bE(\dot{\bu}^{\mu}),\\
&\bH \dot{\bxi} = \mu^{-1}(|\bsig_D^{\mu}-\bxi_D|-\kappa)_+ \frac{\bsig^{\mu}_D-\bxi_D}{|\bsig^{\mu}_D-\bxi_D|}
 \end{aligned}\right\}
 \qquad \textrm{ a.e. in } (0,T)\times \mO. \label{TPM}
\end{equation}
In addition, if the safety initial load condition \eqref{SLMB} holds and $\bsig_0$ and $\bu_0$ satisfy
\begin{equation}
\begin{split}
\label{AAK}
\bsig_0 &\in W^{2,\infty}(0,T; L^2(\mO; \mathbb{R}^{d\times d}_{sym})) \cap W^{1,\infty}(0,T; W^{1,2}(\mO;\mathbb{R}^{d\times d}_{sym})),\\
\bu_0 &\in W^{2,\infty}(0,T; W^{1,2}(\mO; \mathbb{R}^{d\times d}_{sym})) \cap W^{1,\infty}(0,T; W^{2,2}(\mO;\mathbb{R}^{d\times d}_{sym})),
\end{split}
\end{equation}
then we have the following uniform estimates
\begin{equation}\label{unifor-estimates}
\begin{split}
&\sup_{t\in (0,T)} \left(\|\dot\bsig^{\mu}(t)\|_2^2 + \|\dot{\bxi}^{\mu}(t)\|_2+\|\dot{\bu}(t)\|_{1,2}^2\right)\le C(\bu_0,\bsig_0,\mO,T)
\end{split}
\end{equation}
and there exists a sequence that we do not relabel such that
$$
(\bsig^{\mu}, \bxi^{\mu}, \bu^{\mu})\rightharpoonup^* (\bsig,\bxi, \bu)
$$
in the topology induced by the estimate \eqref{unifor-estimates}, where $(\bsig,\bxi, \bu)$ is a solution in a sense of Definition~\ref{D1} and satisfies \eqref{PR11} almost everywhere.

%Moreover, %
%where the constant $C(\tilde{\mO})$ depends only on $\bsig^s$, $\bA$, $\mO$ and $\tilde{\mO}$. Moreover, there exists a subsequence that we do not relabel such that
%\begin{equation}\label{uniform-konvergence}
%\begin{aligned}
%\bsig^{\mu} &\rightharpoonup^* \bsig &&\textrm{in } W^{1,\infty}(0,T; L^2(\mO, \mathbb{R}^{d\times d}_{sym})), \\
%\bsig^{\mu} &\rightharpoonup \bsig &&\textrm{in } N^{\frac{3}{2},2}(0,T; L^2(\mO, \mathbb{R}^{d\times d}_{sym})), \\
%\mu^{-1}(|\bsig_D^{\mu}|-1)_+\frac{\bsig_D^{\mu}}{|\bsig^{\mu}_D|} &\rightharpoonup^* \dbep &&\textrm{in } L^{\infty}(0,T; \mathcal{M}(\overline{\mO}; \mathbb{R}^{d\times d}_{sym})),\\
%\mu^{-1}(|\bsig_D^{\mu}|-1)_+ &\rightharpoonup^* \lambda &&\textrm{in } L^{\infty}(0,T; \mathcal{M}(\overline{\mO})), \\
%\bE(\bu^{\mu}) &\rightharpoonup^* \bE(\bu) &&\textrm{in } W^{1,\infty}(0,T; \mathcal{M}(\overline{\mO}; \mathbb{R}^{d\times d}_{sym})),\\
%\bu^{\mu} &\rightharpoonup^* \bu &&\textrm{in } W^{1,\infty}(0,T; L^{d'}(\mO; \mathbb{R}^d)),
%\end{aligned}
%\end{equation}
%where $\bsig$ is a weak solution in sense of Definition~\ref{D1}, \eqref{KTF} and \eqref{suma} hold  and
%\begin{equation}
%\bA \dot{\bsig} + \dbep= \bE(\dot{\bu}) \qquad \textrm{ in } (0,T)\times \overline{\mO}.\label{TPM-limit}
%\end{equation}
%In addition, there exists $\varepsilon>0$ such that
%\begin{equation}\label{uniformre}
%\bsig\in L^{\infty}(0,T; BMO(\mO; \mathbb{R}^{d\times d}_{sym})), \qquad \dot\bu \in L^{\infty}(0,T;L^{d'+\varepsilon}(\mO;\mathbb{R}^d)).
%\end{equation}
Moreover, if $\mO\in \mathcal{C}^{1,1}$ then for any compact $\tilde{\mO}\subset \mO$, any $\delta>0$ and arbitrary nonnegative $\phi \in \mathcal{C}^{\infty}$ fulfilling $\supp \phi \cap \overline{\partial \mO}_D \cap \overline{\partial \mO}_N=\emptyset$, we have the following estimate
\begin{equation}
\begin{split}
&\sup_{t\in (0,T)} \left( \|\bsig^{\mu}(t)\|_{W^{1,2}(\tilde{\mO})} +\|\bxi^{\mu}(t)\|_{W^{1,2}(\tilde{\mO})} +\|\bu^{\mu}(t)\|_{W^{2,2}(\tilde{\mO})}\right) \\
&+\sup_{t\in (0,T)} \left( \|\phi{\bsig}^{\mu}(t)\|_{N^{1,2}_{\btau}({\mO})} +\|\phi{\bxi}^{\mu}(t)\|_{N^{1,2}_{\btau}({\mO})} +\|\phi\nabla {\bu}^{\mu}(t)\|_{N^{1,2}_{\btau}({\mO})}\right) \\
&+\sup_{t\in (0,T)} \left( \|\phi{\bsig}^{\mu}(t)\|_{N^{\frac35-\delta,2}_{n}({\mO})} +\|\phi{\bxi}^{\mu}(t)\|_{N^{\frac35-\delta,2}_{n}({\mO})}  +\|\phi\nabla {\bu}^{\mu}(t)\|_{N^{\frac35-\delta,2}_{n}({\mO})} \right)\\
& + \|\dot{\bsig}^{\mu}\|_{N^{\frac12,2}(0,T; L^2(\mO))} +\|\dot{\bxi}^{\mu}\|_{N^{\frac12,2}(0,T; L^2(\mO))} +\|\nabla \dot{\bu}^{\mu}\|_{N^{\frac12,2}(0,T; L^2(\mO))}\\
& + \|\dot{\bsig}^{\mu}\|_{L^{2}(0,T; N^{\frac12,2}(\tilde{\mO}))} +\|\dot{\bxi}^{\mu}\|_{L^{2}(0,T; N^{\frac12,2}(\tilde{\mO}))} +\|\nabla \dot{\bu}^{\mu}\|_{L^{2}(0,T; N^{\frac12,2}(\tilde{\mO}))}\\
& + \|\phi\dot{\bsig}^{\mu}\|_{L^{2}(0,T; N^{\frac12,2}_{\btau}(\mO))} +\|\phi\dot{\bxi}^{\mu}\|_{L^{2}(0,T; N^{\frac12,2}_{\btau}(\mO))} +\|\phi\nabla \dot{\bu}^{\mu}\|_{L^{2}(0,T; N^{\frac12,2}_{\btau}(\mO))}\\
& + \|\phi\dot{\bsig}^{\mu}\|_{L^{2}(0,T; N^{\frac15-\delta,2}_{n}(\mO))} +\|\phi\dot{\bxi}^{\mu}\|_{L^{2}(0,T; N^{\frac15-\delta,2}_{n}(\mO))} +\|\phi\nabla \dot{\bu}^{\mu}\|_{L^{2}(0,T; N^{\frac15-\delta,2}_{n}(\mO))}\\
&\le C(\phi,\delta, \tilde{\mO}, \bu_0, \bsig_0),
\label{ident-limit}
\end{split}
\end{equation}
which due to the weak lower semicontinuity holds also for the limit $(\bsig, \bxi, \bu)$.
%where
%$$
%K(t):=\left\{y\in \mO:\; \liminf_{\varepsilon \to 0} \sup_{R\in (0,\varepsilon)} \varepsilon \int_{B_R(y)}R^{1-d-\varepsilon/2}|\dot{\bu}(t,x)|^{1+\varepsilon}\dx \ge 1\right\}.
%$$
%In particular, if $\dot\bu(t)\in L^q(\mO;\mathbb{R}^d)$ for some $q>d$ then
%\begin{equation}\label{sharpF}
%\dbep(t)=\lambda(t) \bsig_D(t) \textrm{ in } \mO.
%\end{equation}
%Consequently, due to \eqref{uniformre}, the identity \eqref{sharpF} always holds for $d=2$.
\end{Theorem}
Please notice here that we used the notations $N^{\alpha,p}$ for the standard Nikoloskii space, $N^{\alpha,p}_{\btau}$ for the space, where we control the tangential differences, i.e., the space of functions whose fractional $\alpha$-th tangential\footnote{Tangential here means in the directions that are orthogonal to the normal vector at  boundary $\partial \mO$.} derivatives belongs to the Lebesgue space $L^p$ and similarly, $N^{\alpha,p}_{\bn}$ for the space, where the $\alpha$-th normal derivative belongs to $L^p$.

For the isotropic hardening, we have the following result.
\begin{Theorem}[Isotropic hardening]\label{TH}
Let all assumptions of Definition~\ref{D1} be satisfied. Then for all $\mu>0$ there exists a unique triple $(\bsig^{\mu},\xi^{\mu},\bu^{\mu})$ such that $\bu^{\mu}-\bu_0\in W^{1,2}(0,T; W^{1,2}_0(\mO;\mathbb{R}^d))$, $\xi^{\mu}(0)=0$, $\bsig^{\mu}=\bsig_0(0)$ and
\begin{equation}
\left. \begin{aligned}
&\bA \dot{\bsig}^{\mu} + \mu^{-1}(|\bsig_D^{\mu}|-\kappa-\xi)_+ \frac{\bsig^{\mu}_D}{|\bsig^{\mu}_D|} = \bE(\dot{\bu}^{\mu}),\\
& H \dot{\xi} = \mu^{-1}(|\bsig_D^{\mu}|-\kappa-\xi)_+
 \end{aligned}\right\}
 \qquad \textrm{ a.e. in } (0,T)\times \mO. \label{TPMHe}
\end{equation}
In addition, if the safety initial load condition \eqref{SLMB} holds and  $\bsig_0$ and $\bu_0$ satisfy
\begin{equation}
\begin{split}
\label{AAKHe}
\bsig_0 &\in W^{2,\infty}(0,T; L^2(\mO; \mathbb{R}^{d\times d}_{sym})) \cap W^{1,\infty}(0,T; W^{1,2}(\mO;\mathbb{R}^{d\times d}_{sym})),\\
\bu_0 &\in W^{2,\infty}(0,T; W^{1,2}(\mO; \mathbb{R}^{d\times d}_{sym})) \cap W^{1,\infty}(0,T; W^{2,2}(\mO;\mathbb{R}^{d\times d}_{sym})),
\end{split}
\end{equation}
then we have the following uniform estimates
\begin{equation}\label{unifor-estimatesHe}
\begin{split}
&\sup_{t\in (0,T)} \left(\|\dot\bsig^{\mu}(t)\|_2^2 + \|\dot{\xi}^{\mu}(t)\|_2+\|\dot{\bu}(t)\|_{1,2}^2\right)\le C(\bu_0,\bsig_0,\mO,T)
\end{split}
\end{equation}
and there exists a sequence that we do not relabel such that
$$
(\bsig^{\mu}, \xi^{\mu}, \bu^{\mu})\rightharpoonup^* (\bsig,\xi, \bu)
$$
in the topology induced by the estimate \eqref{unifor-estimates}, where $(\bsig,\xi, \bu)$ is a solution in a sense of Definition~\ref{D2} and satisfies \eqref{He11} almost everywhere.

Moreover, defining
$$
\alpha:=\frac{2d-7+\sqrt{1+4d^2+20d}}{8(d-1)},
$$
then  for $\mO\in \mathcal{C}^{1,1}$,   any compact $\tilde{\mO}\subset \mO$, any $\delta>0$ and arbitrary nonnegative $\phi \in \mathcal{C}^{\infty}$ fulfilling $\supp \phi \cap \overline{\partial \mO}_D \cap \overline{\partial \mO}_N=\emptyset$, we have the following estimate
\begin{equation}
\begin{split}
&\sup_{t\in (0,T)} \left( \|\bsig^{\mu}(t)\|_{W^{1,2}(\tilde{\mO})} +\|\xi^{\mu}(t)\|_{W^{1,2}(\tilde{\mO})} \right) \\
&\sup_{t\in (0,T)} \left( \|\phi{\bsig}^{\mu}(t)\|_{N^{1,2}_{\btau}({\mO})} +\|\phi{\xi}^{\mu}(t)\|_{N^{1,2}_{\btau}({\mO})} \right) \\
&\sup_{t\in (0,T)} \left( \|\phi{\bsig}^{\mu}(t)\|_{N^{\alpha-\delta,2}_{n}({\mO})} +\|\phi{\xi}^{\mu}(t)\|_{N^{\alpha-\delta,2}_{n}({\mO})}   \right)\\
&\quad + \|\dot{\bsig}^{\mu}\|_{N^{\frac12,2}(0,T; L^2(\mO))} +\|\dot{\xi}^{\mu}\|_{N^{\frac12,2}(0,T; L^2(\mO))} \\
&\quad + \|\dot{\bsig}^{\mu}\|_{L^{2}(0,T; N^{\frac12,2}(\tilde{\mO}))} +\|\dot{\xi}^{\mu}\|_{L^{2}(0,T; N^{\frac12,2}(\tilde{\mO}))} \\
&\quad + \|\phi\dot{\bsig}^{\mu}\|_{L^{2}(0,T; N^{\frac12,2}_{\btau}(\mO))} +\|\phi\dot{\xi}^{\mu}\|_{L^{2}(0,T; N^{\frac12,2}_{\btau}(\mO))} \\
&\quad + \|\phi\dot{\bsig}^{\mu}\|_{L^{2}(0,T; N^{\frac{\alpha}{3}-\delta,2}_{n}(\mO))} +\|\phi\dot{\xi}^{\mu}\|_{L^{2}(0,T; N^{\frac{\alpha}{3}-\delta,2}_{n}(\mO))}\\
&\le C(\phi,\delta, \tilde{\mO}, \bu_0, \bsig_0),
\label{ident-limitHe}
\end{split}
\end{equation}
which due to the weak lower semicontinuity holds also for the limit $(\bsig, \xi, \bu)$. In addition, for any nonnegative  $\varphi\in \mathcal{C}^{\infty}(\mathbb{R})$ that fulfils $\supp \varphi \cap \partial \mO_N$, there holds
\begin{equation}
\begin{split}
&\sup_{t\in (0,T)} \left( \|\varphi{\bsig}^{\mu}(t)\|_{N^{\frac35-\delta,2}_{n}({\mO})} +\|\varphi{\xi}^{\mu}(t)\|_{N^{\frac35-\delta,2}_{n}({\mO})}   \right)\\
&\quad + \|\varphi\dot{\bsig}^{\mu}\|_{L^{2}(0,T; N^{\frac15-\delta,2}_{n}(\mO))} +\|\varphi\dot{\xi}^{\mu}\|_{L^{2}(0,T; N^{\frac15-\delta,2}_{n}(\mO))}\\
&\le C(\delta, \bu_0, \bsig_0, \varphi),
\label{ident-limitHe2}
\end{split}
\end{equation}

\end{Theorem}

To end this part of the paper, we emphasize the essential novelties stated in Theorems~\ref{TPR}--\ref{TH}. The existence of a solution was already established in \cite{Jo76,Jo78} and there is nothing new in the paper. Also the interior $W^{1,2}$ regularity for the stress has been proven in \cite{Se92,Se94} for various models. A key improvement concerning the interior regularity is due to \cite{FrSp12} (see also the related paper \cite{FrSp12a} for problems without hardening), where the authors showed\footnote{In fact they showed an estimate, from which one can deduce the result following the method invented in \cite{FrSc15}.} that $\dot{\bsig}$ and $\dot{\bxi}$ belongs to $L^{2}(0,T; N^{\frac12,2}_{loc}(\mO))$ and $N^{\frac12,2}(0,T; L^2(\mO))$ (see also \cite{AlNe09}, where a weaker result is obtained for similar problems). However, their estimates were true only for the limit, i.e., for the solution, but were not uniform with respect to approximation. Our result overcomes this weakness and we are able to obtain the uniform $\mu$-independent estimates, see also \cite{BuFr18}, where the similar statement was proven for plasticity without hardening, or also \cite{Kn06} for  up to the boundary or \cite{BaMo18} for $W^{1,2}_{loc}$ results for various models in elasto-plasticity theory without hardening.

Concerning the results up to the boundary, the tangential regularity for $\bsig$ and $\bxi$ was already obtained in \cite{FrLo09,FrLo11} and the authors also obtained that the solution belongs to $L^{\infty}(0,T; N_{\bn}^{\frac12+\delta,2}(\mO))$. Hence, our result significantly improves this estimate since in case of kinematic hardening or in case of Dirichlet data we have $N_{\bn}^{\frac35 - \delta,2}(\mO)$ independently of dimension. In addition for isotropic case and Neumann data, we can precisely trace the improvement as stated in \eqref{ident-limitHe}. Finally, and this is the main improvement, we are able to obtain also the fractional $N^{\frac12,2}_{loc}$ and $N^{\frac12,2}_{\btau}$ regularity for $\dot{\bsig}$ and $\dot{\bxi}$ and even more we have an information $N^{\frac15-\delta,2}_{\bn}$ in normal direction, which is obtained by the cross interpolation. Note that in view of the counterexamples to $W^{1,2}$ regularity up to the boundary proven e.g. in~\cite{Se99,Kn10}, such estimates seems to be optimal.

In the rest of the paper, we will focus only on the kinematic hardening and we shall just emphasize where are the differences. Obviously, in the kinematic hardening case, we can transfer the obtained regularity from $\dot{\bsig}$ and $\dot{\bxi}$ to $\bE(\dot{\bu})$ just by using the equation. Then the regularity for $\nabla \dot{\bu}$ just follows from the Korn inequality applied on $\mO$ or its sub-domain and we do not provide details here. On the other hand, in case of isotropic hardening, we are not able to use such a procedure. The best, we can do is just to transfer better integrability to $\nabla \dot{\bu}$, see also \cite{FrLo09,FrLo11}, but this is also omitted here, since it is just direct consequence of the Korn inequality and embedding theorem.  The second case is when we combine the isotropic hardening and the Neumann boundary conditions. The reason for that is that in such case we cannot use any version of anisotropic Korn inequality to transfer optimal anisotropic integrability from symmetric gradient to the full gradient. Also, we would like to emphasize that our restriction on the constant $\bA$ and $\bH$ is not necessary and the proof would remain almost exactly identical if they are Lipschitz functions of $(t,x)$. Finally, to simplify the presentation, we consider only the flat boundary, however for $\mathcal{C}^{1,1}$ boundaries, it is just a technical difficulty  to transform the problem with general boundary to the case of flat boundary. Finally, let us remark that we cannot avoid a possible singularity on $\overline{\partial \mO}_D \cap \overline{\partial \mO}_N$ from principal reasons, since even for linear elliptic problems one may observe a singularity.

\section{Proof of Theorems~\ref{TPR} and \ref{TH}}
We focus here mainly on the kinematic hardening case, since the proof for the isotropic hardening is very similar. Only on certain places, we discuss the possible differences. Also to simplify notation, we set $\kappa\equiv 1$ in what follows. Finally having a~priori estimates stated in Theorem~\ref{TPR} or in Theorem~\ref{TH}, it is very classical to pass to the limit $\mu \to 0_+$ and to obtain the solution to the original problem, i.e., to the kinematic hardening and the isotropic hardening problems. Therefore, we also skip the limiting passage in the proof. Finally, we do not discuss the problem of existence of a solution for fix $\mu>0$ since it was already established by many authors, see e.g. \cite{Jo76,Jo78}, but we rather focus on a~priori estimates. Also in order to shorten the text, we omit writing superscripts to emphasize we deal with a solution to an approximative problem. Furthermore, we do not trace the dependence of all constants on $\mO$ or $C_1$ and in what follows the constant $C$ has a meaning of some universal generic constant that may vary line to line but is independent of $\mu$. In case we want to emphasize the dependence of this constant on some parameter it is clearly denoted.

%As mentioned already in the introduction, we focus only on the Prandtl--Reuss model here and we shall use  the Per\v{c}ina approximation. We would like to notice here that a very similar procedure was developed in \cite{Te86} with a slightly different approximation - the Norton-Hoff approximation. Also to simplify the presentation, we shall consider in what follows that $\kappa\equiv 1$.

\subsection{First a~priori uniform estimates}\label{ss1}
Thus, we shall assume that for any $\mu>0$ there exists a solution \eqref{TPM}. The existence of such a $\sigma$ can be shown e.g. by the Rothe approximation and we refer the interested reader to \cite{FrSp12} or \cite{St03}, where even a more difficult case of problem without hardening is treated, or to original papers~\cite{Jo76,Jo78}. Hence, we assume that there is $\bu \in W^{1,2}(0,T; W^{1,2}(\mO; \mathbb{R}^d))$ such that for all $t\in (0,t)$ $\bu-\bu_0 \in \mathcal{V}$ and $\bu(0)=\bu_0(0)$, and that there is $(\bsig,\bxi)\in \mathcal{F}_{el}$ fulfilling
\begin{equation}
\label{vztha}
\left. \begin{aligned}
&\bA \dot{\bsig} + \mu^{-1}(|\bsig_D-\bxi_D|-1)_+ \frac{\bsig_D-\bxi_D}{|\bsig_D-\bxi_D|} = \bE(\dot \bu),\\
&\bH \dot{\bxi}= \mu^{-1}(|\bsig_D-\bxi_D|-1)_+ \frac{\bsig_D-\bxi_D}{|\bsig_D-\bxi_D|}
\end{aligned} \right\}\qquad \textrm{ in } (0,T)\times \mO.
\end{equation}
The next step is to derive the uniform ($\mu$ independent estimates) for $(\bu, \bsig, \bxi)$. We proceed here formally, since the estimates are known, see e.g. \cite{Te86,St03,FrSp12}. Taking the scalar product of the first equation in \eqref{vztha} with $\bsig-\bsig_0$, recall here that $\bsig_0$ satisfies the compatibility condition \eqref{SL}, we deduce after integration over $\mO$ that
\begin{equation}
\label{vztha2}
\begin{split}
&\int_{\mO}\bA (\dot{\bsig} -\dot{\bsig}_0)\cdot (\bsig - \bsig_0) + \mu^{-1}(|\bsig_D-\bxi_D|-1)_+ \frac{(\bsig_D-\bxi_D)\cdot (\bsig_D - \bsig_{0D})}{|\bsig_D-\bxi_D|}\ddd x \\
&\quad = \int_{\mO}\bE(\dot \bu - \dot \bu_0)\cdot (\bsig-\bsig_0) \ddd x + \int_{\mO} (\bE (\dot \bu_0) -\bA \dot{\bsig}_0)\cdot (\bsig - \bsig_0)\ddd x.
\end{split}
\end{equation}
Second identity, we deduce from \eqref{vztha} by taking the scalar product with $\xi_0$. Thus, we have
\begin{equation}
\label{vztha2.5}
\begin{split}
&\int_{\mO}\bH (\dot{\bxi} -\dot{\bxi}_0)\cdot (\bxi - \bxi_0)\ddd x \\
&\quad = \int_{\mO} \mu^{-1}(|\bsig_D-\bxi_D|-1)_+ \frac{(\bsig_D-\bxi_D)\cdot (\bxi_D - \bxi_{0D})}{|\bsig_D-\bxi_D|}\ddd x -\dot{\bxi}_0\cdot (\bxi - \bxi_0)\ddd x.
\end{split}
\end{equation}
Since $(\bsig(t),\bxi)\in \mathcal{F}_{el}$ and $(\bu-\bu_0)\in \mathcal{V}$, we see that the first term on the right hand side of \eqref{vztha} vanishes. In addition, since $(\bsig_0,\bxi_0)$ satisfies the compatibility condition \eqref{SL}, we observe
$$
\begin{aligned}
&(|\bsig_D-\bxi_D|-1)_+(\bsig_D-\bxi_D) \cdot (\bsig_D -\bsig_{0D} - \bxi_D + \bxi_{0D})\\
&=(|\bsig_D-\bxi_D|-1)_+(|\bsig_D-\bxi_D|^2 - (\bsig_D-\bxi_D) \cdot (\bsig_{0D}-\bxi_{0D}))\\
&\ge (|\bsig_D-\bxi_D|-1)_+|\bsig_D-\bxi_D| (|\bsig_D-\bxi_D| - |\bsig_{0D}-\bxi_{0D}|) \ge 0.
\end{aligned}
$$
Finally, we can sum \eqref{vztha2} and \eqref{vztha2.5} and by using the above inequality, the symmetry of  $\bA$ and $\bH$ as well as the ellipticity \eqref{ellip} and also the boundedness of $\bA$ and $\bH$, we arrive to the inequality
\begin{equation}
\begin{split}
&\frac{d}{dt}\int_{\mO}\bA (\bsig -\bsig_0)\cdot (\bsig - \bsig_0)+\bH (\bxi -\bxi_0)\cdot (\bxi - \bxi_0)\ddd x \\
&\quad \le C\left(\|\bE (\dot \bu_0)\|_2^2+ \|\dot{\bsig}_0\|_2^2 + \|\dot{\bxi}_0\|_2^2 \right)\\
&\qquad + C\int_{\mO}\bA (\bsig -\bsig_0)\cdot (\bsig - \bsig_0)+\bH (\bxi -\bxi_0)\cdot (\bxi - \bxi_0)\ddd x.
\end{split}\label{nove}
\end{equation}
Consequently, by the Gronwall lemma, we have
\begin{equation}
\label{vztha3}
\begin{split}
&\sup_{t\in (0,T)} (\|\bsig(t)\|_2^2 + \|\bxi(t)\|_2^2) \\
&\quad \le C\left(\|\bsig_0(0)\|_2^2 + \int_0^T\|\bE (\dot \bu_0)\|_2^2 + \|\dot{\bsig}_0\|_2^2 + \|\dot{\bxi}_0\|_2^2\ddd t\right)  \le C,
\end{split}
\end{equation}
where the last inequality follows from the assumptions on data (namely on $\bsig_0$, $\bu_0$ and $\bxi_0$).

The next step is to test the first equation in  \eqref{vztha} by $\dot{\bsig}-\dot{\bsig}_0$ and the second equation by $\dot{\bxi}$. Doing so, and summing the resulting identities,  we get
\begin{equation}
\label{vztha4}
\begin{split}
&\int_{\mO}\bA \dot{\bsig} \cdot \dot{\bsig}+ \bH \dot{\bxi} \cdot \dot{\bxi} + \mu^{-1}(|\bsig_D-\bxi_D|-1)_+ \frac{(\bsig_D-\bxi_D)\cdot (\dot{\bsig}_D-\dot{\bxi}_D)}{|\bsig_D-\bxi_D|}\ddd x \\
&= \int_{\mO}\bE(\dot \bu - \dot \bu_0)\cdot (\dot{\bsig}-\dot{\bsig}_0) + \mu^{-1}(|\bsig_D-\bxi_D|-1)_+ \frac{(\bsig_D-\bxi_D)\cdot \dot{\bsig}_{0D}}{|\bsig_D-\bxi_D|}\ddd x\\
&\qquad + \int_{\mO} (\bE (\dot \bu_0) \cdot (\dot\bsig - \dot\bsig_0) + \bA \dot{\bsig} \cdot \dot{\bsig}_0 \ddd x\\
&= \int_{\mO}\bH \dot{\bxi} \cdot \dot{\bsig}_{0D}+ (\bE (\dot \bu_0) \cdot (\dot\bsig - \dot\bsig_0) + \bA \dot{\bsig} \cdot \dot{\bsig}_0 \ddd x,
\end{split}
\end{equation}
where the last equality follows from \eqref{vztha}$_2$ and the fact that $(\bsig,\bxi)\in \mathcal{F}_{el}$. Using the ellipticity \eqref{ellip}, the Young inequality and the following identity
$$
\mu^{-1}(|\bsig_D-\bxi_D|-1)_+ \frac{(\bsig_D-\bxi_D)\cdot (\dot{\bsig}_D-\dot{\bxi}_D)}{|\bsig_D-\bxi_D|}=\frac{1}{2} \frac{\partial}{\partial t} \mu^{-1}(|\bsig_D-\bxi_D|-1)_+^2,
$$
we see that it follows from \eqref{vztha4} that
\begin{equation*}
\begin{split}
\frac{d}{dt} \int_{\mO}\mu^{-1}(|\bsig_D-\bxi_D|-1)^2_+\ddd x +C_1(\|\dot{\bsig}\|_2^2+\|\dot{\bxi}\|_2^2)&\le  C(\|\bE (\dot \bu_0)\|_2^2 + \|\dot\bsig_0\|_2^2).
\end{split}
\end{equation*}
Thus, integrating with  respect to $t\in (0,T)$ and using the fact that $\bxi(0)=0$ and that $|\bsig(0)|=|\bsig_0(0)|\le 0$, we get the uniform bound
\begin{equation}
\begin{split}\label{vztha7}
&\sup_{t\in (0,T)}\int_{\mO}\mu^{-1}(|\bsig_D-\bxi_D|-1)^2_+\ddd x +\int_0^T\|\dot{\bsig}\|_2^2+ \|\dot{\bxi}\|_2^2 \ddd t\\
&\qquad \le C\int_0^T\|\bE (\dot \bu_0)\|_2^2 + \|\dot\bsig_0\|_2^2\ddd t \le C,
\end{split}
\end{equation}
where the last inequality (with $C$ being independent of $\mu$) follows from the assumptions on $\bu_0$ and  $\bsig_0$. In addition, it follows from \eqref{vztha} that
$$
|\bE(\dot \bu)|\le C(|\dot{\bsig}| + |\dot{\bxi}|)
$$
and consequently, \eqref{vztha7}, the Korn inequality and the assumptions on $\bu_0$ leads to the uniform bound
\begin{equation}\label{unifG}
\int_0^T \|\dot \bu\|_{1,2}^2 \ddd t \le C.
\end{equation}

\bigskip

The last step is $L^{\infty}$ bound for the time derivative. We apply the time derivative to \eqref{vztha} and take the scalar product of the first equation with $\dot \bsig-\dot \bsig_0$ and the scalar product of the second equation with $\dot{\bxi}$. Summing these to identities we obtain
\begin{equation}
\label{vztha44}
\begin{split}
&\int_{\mO}\bA (\ddot{\bsig}-\ddot{\bsig}_0) \cdot (\dot{\bsig}-\dot{\bsig}_0)+ \bH \ddot{\bxi} \cdot \dot{\bxi} \ddd x\\
&\qquad + \int_{\mO}\mu^{-1}\frac{\partial(|\bsig_D-\bxi_D|-1)_+ \frac{\bsig_D-\bxi_D}{|\bsig_D-\bxi_D|}}{\partial t}\cdot (\dot{\bsig}_D-\dot{\bxi}_D - \dot{\bsig}_{0D})\ddd x \\
&\qquad = \int_{\mO}\bE(\ddot \bu - \ddot \bu_0)\cdot (\dot{\bsig}-\dot{\bsig}_0) \ddd x + \int_{\mO} (\bE (\ddot \bu_0)-\bA \ddot{\bsig}_0) \cdot (\dot\bsig - \dot\bsig_0)\ddd x.
\end{split}
\end{equation}
The first term on the right hand side vanishes and for the part of the second  term on the left hand side we use the following estimate
\begin{equation}\label{reco}
\begin{aligned}
\mu^{-1}\frac{\partial(|\bbeta|-1)_+ \frac{\bbeta}{|\bbeta|}}{\partial t}\cdot \dot{\bbeta}=\frac{\mu^{-1}\chi_{|\bbeta|>1}}{|\bbeta|}\left(|\dot \bbeta|^2 (|\bbeta|-1)+|\partial_t |\bbeta||^2\right)\ge 0.
\end{aligned}
\end{equation}
Consequently, using the H\"{o}lder inequality and the above inequality, we see that \eqref{vztha44} implies (using also the second identity in \eqref{vztha})
\begin{equation}
\label{vztha45}
\begin{split}
&\frac12 \frac{d}{dt}\int_{\mO}\bA (\dot{\bsig}-\dot{\bsig}_0) \cdot (\dot{\bsig}-\dot{\bsig}_0)+ \bH \dot{\bxi}\cdot \dot{\bxi}  - 2\bH \dot{\bxi} \cdot \dot{\bsig}_{0D}\ddd x\\
&\qquad \le C(\|\bE (\ddot \bu_0)\|_2+\|\ddot{\bsig}_0\|_2)(1+\|\dot{\bxi}\|_2+ \|\dot\bsig - \dot\bsig^s\|_2).
\end{split}
\end{equation}
Hence, adding  the term
$$
\frac12 \frac{d}{dt}\int_{\mO}\bH \dot{\bsig}_{0D} \cdot \dot{\bsig}_{0D}\ddd x
$$
to both sides of \eqref{vztha45}, we deduce
\begin{equation}
\label{vztha45.1}
\begin{split}
&\frac12 \frac{d}{dt}\int_{\mO}\bA (\dot{\bsig}-\dot{\bsig}_0) \cdot (\dot{\bsig}-\dot{\bsig}_0)+ \bH (\dot{\bxi}-\dot{\bsig}_{0D})\cdot (\dot{\bxi}-\dot{\bsig}_{0D})\ddd x\\
&\qquad \le C(\|\bE (\ddot \bu_0)\|_2+\|\ddot{\bsig}_0\|_2)(1+\|\dot{\bsig}_{0D}\|_2+\|\dot{\bxi}-\dot{\bsig}_{0D}\|_2+ \|\dot\bsig - \dot{\bsig}_{0D}\|_2).
\end{split}
\end{equation}
Consequently, integration of this inequality and the ellipticity assumption \eqref{ellip} lead to the estimate
\begin{equation}\label{infty1}
\begin{split}
\sup_{t\in (0,T)}& \left(\|\dot\bsig(t)\|_2^2 + \|\dot{\bxi}(t)\|_2^2\right)\\
&\le C\left(1+ \|\dot{\bsig}_0(0)\|_2^2 + \left(\int_0^T \|\ddot{\bu}_0\|_{1,2} + \|\ddot{\bsig}_{0}\|_2\ddd t\right)^2\right) \le C.
\end{split}
\end{equation}
Furthermore, it follows from \eqref{vztha} and the Korn inequality that
\begin{equation}\label{infty2}
\sup_{t\in (0,T)} \|\dot{\bu}(t)\|_{1,2}^2\le C.
\end{equation}

\subsection{Uniform  $W^{1,2}$ estimates} \label{ss2} In this subsection, we  derive the uniform interior estimates on $\nabla \bsig$ and on $\nabla \bxi$ and estimates for tangential derivatives of $(\bsig, \bxi)$ up to the boundary.

To simplify the presentation, we consider here only a flat boundary case but it can be straightforwardly extended to the general boundary. Also since the interior regularity is in fact easier to prove than the boundary regularity, we provide here only the estimates near the boundary for tangential derivatives.  Hence to simplify the notation, we assume from now the most difficult case, i.e., we focus on a cube  $(-1,1)^{d-1}\times (0,1) \subset \mO$, where the Dirichlet and the Neumann parts are supposed to satisfy
$$
 (-1,1)^{d-2}\times (-1,0)\times \{0\} \subset \partial\mO_D \qquad (-1,1)^{d-2}\times (0,1)\times \{0\} \subset \partial\mO_N.
$$
Our goal is to show that except the set $(-1,1)^{d-2} \times \{0\}\times \{0\}$ we have uniform estimates for $D_j \bsig$ and $D_j\bxi$ in the space $L^2$ for all $j=1,\ldots, d-1$, where $D_j$ denotes the partial derivative with respect to $x_j$.

Thus, let $\phi\in
\mathcal{D}(-1,1)^d$ be arbitrary nonnegative function satisfying $\phi\le 1$. Furthermore, we require that for some $\varepsilon_0>0$, the function $\phi$  satisfies for all $x_1,\ldots, x_{d-1}$ and all $|s|+|t|\le \varepsilon_0$ that $\tau(x_1,\ldots, x_{d-2},s,t)=0$.   Next, we fix arbitrary $j=1,\ldots, d-1$ and  apply the operator $D_j$ to both equations in \eqref{vztha}. Then we take the scalar product of the first equation with $D_j (\bsig-\bsig_0)\phi^2$ and the scalar product of the second equation with $D_j \bxi \phi^2$, sum the results and integrate over $\mO$ to deduce the identity
\begin{equation}
\begin{split}
&\int_{\mO}\bA D_j (\dot\bsig -\dot{\bsig}_0) \cdot D_j( \bsig-\bsig_0) \phi^{2} +\bH D_j \dot{\bxi} \cdot D_j \bxi \phi^2 \ddd x\\
 &\quad + \int_{\mO} \mu^{-1}D_j
\left(\frac{(|\bsig_D-\bxi_D| - 1)_+ \bsig_D-\bxi_D}{|\bsig_D-\bxi_D|}\right)\cdot  D_j (\bsig_D - \bxi_D)
\phi^{2}\ddd x \\
&= \int_{\mO}D_j \bE(\dot\bu)\cdot D_j (\bsig -\bsig_0) \phi^{2} - \bA D_j \dot{\bsig}_0 \cdot  D_j( \bsig-\bsig_0) \phi^{2} \ddd x.
\label{app1}
\end{split}
\end{equation}
Note that the second integral on the left hand side is nonnegative (see the same procedure as for the estimates for the first time derivatives) and can be neglected. Next, using the symmetry of $\bA$ and $\bH$, we can deduce that
\begin{equation}
\begin{split}
\frac12 \frac{d}{dt}&\int_{\mO}\bA D_j (\bsig-\bsig_0) \cdot D_j (\bsig-\bsig_0)\phi^2 + \bH D_j\bxi \cdot D_j\bxi \phi^2\ddd x   \\
&\le  \int_{\mO}\bE(D_j \dot\bu- D_j \dot{\bu}_0)\cdot D_j (\bsig -\bsig_0) \phi^{2}\ddd x\\
&\quad +\int_{\mO}(\bE(D_j \dot{\bu}_0) - \bA D_j \dot{\bsig}_0 ) \cdot  D_j( \bsig-\bsig_0) \phi^{2}\\
&\le  \int_{\mO}\nabla (\phi^2(D_j \dot\bu- D_j \dot{\bu}_0))\cdot D_j (\bsig -\bsig_0)\ddd x\\
 &\quad -2\int_{\mO}\nabla \phi\otimes (D_j \dot\bu- D_j \dot{\bu}_0))\cdot  D_j (\bsig -\bsig_0)\phi \ddd x\\
&\qquad  +C(\|\dot{\bu}_0\|_{2,2}+ \|\dot{\bsig}\|_{1,2}) \| D_j( \bsig-\bsig_0) \phi\|_2\\
&\le C(\|\dot{\bu}_0\|_{2,2}+\|\dot{\bu}\|_{1,2}+ \|\dot{\bsig}\|_{1,2}) \| D_j( \bsig-\bsig_0) \phi\|_2,\label{app2}
\end{split}
\end{equation}
where for the last inequality we used integration by parts and the fact that $D_j$ is the tangential derivative and so $D_j(\dot{\bu}-\dot{\bu}_0)$ vanishes on $\partial \mO_D$ and similarly $D_j (\bsig -\bsig_0) \bn=0$ on $\partial \mO_N$ and also thanks to the fact that $\diver D_j (\bsig -\bsig_0)=0$. Consequently, it directly follows from \eqref{app2} that
\begin{equation}\label{odhadnab}
\begin{split}
&\sup_{t\in (0,T)}(\|D_j \bsig(t) \phi\|_2 + \|\phi D_j \bxi(t)\|_2) \le C,
\end{split}
\end{equation}
where we used the a~priori estimate \eqref{infty1} and the assumptions on $\bu_0$ and $\bsig_0$.

\subsection{Fractional time regularity for $\dot{\bsig}$ and $\dot{\bxi}$}\label{ss4}

In this section, we prove the first new result. Although the estimate is known for the solution of the original problem, see \cite{FrSp12}, it was not clear whether the estimate can be obtained uniformly with respect to the parameter $\mu$. In addition, in \cite{FrSp12}, the $N^{\frac12 -\delta,2}$ regularity is proven, while here we obtain $1/2$ derivative estimate.

For arbitrary $w$, we denote its times shift as  $\Delta^{\tau}_t w(t,x):=w(t+\tau,x)-w(t,x)$ and with the help of this notation, we take the scalar product of the first equation in \eqref{vztha} with $-\Delta_t^\tau (\dot{\bsig}-\dot{\bsig}_0)$, and the scalar product  of the second  equation in \eqref{vztha} with $-\Delta_t^\tau \dot{\bxi}$, sum the resulting equalities and finally integrate the result over $\mO$ to get
\begin{equation*}
\begin{split}
&-\int_{\mO} \bA\dot{\bsig}\cdot\Delta_t^{\tau} (\dot{\bsig}-\dot{\bsig}_0) + \bH \dot{\bxi} \cdot \Delta_t^\tau \dot{\bxi} \ddd x\\
&\qquad -\int_{\mO} \mu^{-1}(|\bsig_D-\bxi_D|-1)_+ \frac{\bsig_D-\bxi_D}{|\bsig_D-\bxi_D|} \cdot \Delta_t^{\tau} (\dot{\bsig}_D-\dot{\bxi}_D)\ddd x \\
&= \int_{\mO} \bE (\dot{\bu}_0-\dot{\bu})\cdot \Delta_t^{\tau} (\dot{\bsig}-\dot{\bsig}_0) - \bE(\dot{\bu}_0)\cdot \Delta_t^{\tau} (\dot{\bsig}-\dot{\bsig}_0)- \bH \dot{\bxi} \cdot \Delta_t^{\tau}\dot{\bsig}_{0D}\ddd x.
\end{split}
\end{equation*}
The first term on the right hand side vanishes and after the use of the H\"{o}lder inequality and the bound \eqref{ellip}, and after reorganisation of all terms, we deduce that
\begin{equation}
\label{Ni1}
\begin{split}
&-\int_{\mO} \bA\dot{\bsig}\cdot\Delta_t^{\tau} \dot{\bsig} + \bH \dot{\bxi} \cdot \Delta_t^\tau \dot{\bxi} \ddd x\\
&\qquad -\int_{\mO} \mu^{-1}(|\bsig_D-\bxi_D|-1)_+ \frac{\bsig_D-\bxi_D}{|\bsig_D-\bxi_D|} \cdot \Delta_t^{\tau} (\dot{\bsig}_D-\dot{\bxi}_D)\ddd x \\
&\le C(\|\dot{\bsig}\|_2 + \|\dot{\bxi}\|_2) \|\Delta_t^{\tau}\dot{\bsig}_{0}\|_2 + \int_{\mO} \bE(\dot{\bu}_0)\cdot \Delta_t^{\tau} (\dot{\bsig}_0-\dot{\bsig})\ddd x.
\end{split}
\end{equation}
Next, we focus on the first term on the left hand side. Due to the symmetry of $\bA$ and $\bH$, we have
$$
\begin{aligned}
-\bA\dot{\bsig}\cdot \Delta_t^{\tau} \dot{\bsig}&=\frac12 \bA\Delta_t^{\tau}\dot{\bsig}\cdot \Delta_t^{\tau} \dot{\bsig}-\frac12 \Delta_t^s \left(\bA \dot{\bsig}\cdot \dot{\bsig}\right),\\
-\bH\dot{\bxi}\cdot \Delta_t^{\tau} \dot{\bxi}&=\frac12 \bH\Delta_t^{\tau}\dot{\bxi}\cdot \Delta_t^{\tau} \dot{\bxi}-\frac12 \Delta_t^s \left(\bH \dot{\bxi}\cdot \dot{\bxi}\right).
\end{aligned}
$$
Therefore, substituting these identities into \eqref{Ni1} and using the ellipticity of $\bA$ and $\bH$ (see \eqref{ellip}), we further observe with the help of the a~piori estimate~\eqref{infty1} that
\begin{equation}
\label{Ni2}
\begin{split}
&\int_{\mO} C_1(|\Delta_t^{\tau}\dot{\bsig}|^2+|\Delta_t^{\tau}\dot{\bxi}|^2)  - 2\mu^{-1}(|\bsig_D-\bxi_D|-1)_+ \frac{\bsig_D-\bxi_D}{|\bsig_D-\bxi_D|} \cdot \Delta_t^{\tau} (\dot{\bsig}_D-\dot{\bxi}_D)\ddd x \\
&\quad\le \int_{\mO}  2\bE(\dot{\bu}_0)\cdot \Delta_t^{\tau} (\dot{\bsig}_0-\dot{\bsig})+\Delta_t^{\tau} \left(\bA \dot{\bsig}\cdot \dot{\bsig}+\bH \dot{\bxi}\cdot \dot{\bxi}\right)\ddd x+ C\|\Delta_t^{\tau} \dot{\bsig}_0\|_{2}.
\end{split}
\end{equation}
Next, we integrate the resulting inequality with respect to $\tau$ over the interval $(0,h)$ and with respect to $t$ over the interval $(0,T-h)$ to get
\begin{equation}
\label{Ni3}
\begin{split}
&C_1\int_{0}^{T-h}\int_0^h \|\Delta_t^{\tau}\dot{\bsig}\|^2_2 +\|\Delta_t^{\tau}\dot{\bxi}\|^2_2  \ddd \tau \ddd t \\
&\le \int_{0}^{T-h}\int_{\mO} 2\mu^{-1}(|\bsig_D-\bxi_D|-1)_+ \frac{\bsig_D-\bxi_D}{|\bsig_D-\bxi_D|} \times \\
&\qquad {} \qquad  \times \left(\int_0^h \Delta_t^{\tau}(\dot{\bsig}_D-\dot{\bxi}_D)\ddd \tau\right)\ddd x \ddd t \\
&\quad+ \int_{0}^{T-h}\int_0^h \int_{\mO}  2\bE(\dot{\bu}_0)\cdot \Delta_t^{\tau} (\dot{\bsig}_0-\dot{\bsig})+\Delta_t^{\tau} \left(\bA \dot{\bsig}\cdot \dot{\bsig}+\bH \dot{\bxi}\cdot \dot{\bxi}\right)\ddd x\ddd \tau \ddd t\\
&\qquad  +C\int_{0}^{T-h}\int_0^h \|\Delta_t^{\tau} \dot{\bsig}_0\|_{2}\ddd \tau \ddd t.
\end{split}
\end{equation}
Our goal is to provide uniform estimates for all terms on the right hand side. We start with the easiest one. Hence, using the assumption on $\bsig_0$, we have
\begin{equation}
\label{Ni4}
\begin{split}
\int_{0}^{T-h}\int_0^h  \|\Delta_t^{\tau} \dot{\bsig}_0\|_{2}\ddd \tau \ddd t &=\int_0^h \tau \left(\int_{0}^{T-h}  \frac{\|\Delta_t^{\tau} \dot{\bsig}_0\|_{2}}{\tau} \ddd t \right) \ddd \tau \\
&\le \int_0^h \tau \int_{0}^{T} \| \ddot{\bsig}_0\|_{2} \ddd t \ddd \tau\le Ch^2.
\end{split}
\end{equation}
Next, we focus on the term with $\bE(\dot{\bu}_0)$. We simply shift the differences to $\bu_0$ and then use the assumptions on $\bu_0$ and alreday obtained estimates. In more details, we reorganise the first term as follows
\begin{equation*}
%\label{Ni5}
\begin{split}
&\int_{0}^{T-h}\int_0^h \int_{\mO}  2\bE(\dot{\bu}_0)\cdot \Delta_t^{\tau} (\dot{\bsig}_0-\dot{\bsig})\ddd x \ddd \tau \ddd t\\
&=\int_0^h \int_{0}^{T-h} \int_{\mO}  2\bE(\dot{\bu}_0)\cdot (\dot{\bsig}_0(t+\tau)-\dot{\bsig}(t+\tau))\ddd x  \ddd t \ddd \tau \\
&\quad - \int_0^h \int_{0}^{T-h} \int_{\mO}  2\bE(\dot{\bu}_0)\cdot (\dot{\bsig}_0(t)-\dot{\bsig}(t))\ddd x  \ddd t \ddd \tau\\
&=\int_0^h \int_{\tau}^{T-h+\tau} \int_{\mO}  2\bE(\dot{\bu}_0(t-\tau))\cdot (\dot{\bsig}_0(t)-\dot{\bsig}(t))\ddd x  \ddd t \ddd \tau \\
&\quad - \int_0^h \int_{0}^{T-h} \int_{\mO}  2\bE(\dot{\bu}_0(t))\cdot (\dot{\bsig}_0(t)-\dot{\bsig}(t))\ddd x  \ddd t \ddd \tau\\
&=\int_0^h \int_{\tau}^{T-h} \int_{\mO}  2\bE(\dot{\bu}_0(t-\tau)-\dot{\bu}_0(t))\cdot (\dot{\bsig}_0(t)-\dot{\bsig}(t))\ddd x  \ddd t \ddd \tau \\
&\quad +\int_0^h \int_{T-h}^{T-h+\tau} \int_{\mO}  2\bE(\dot{\bu}_0(t-\tau))\cdot (\dot{\bsig}_0(t)-\dot{\bsig}(t))\ddd x  \ddd t \ddd \tau\\
&\quad - \int_0^h \int_{0}^{\tau} \int_{\mO}  2\bE(\dot{\bu}_0(t))\cdot (\dot{\bsig}_0(t)-\dot{\bsig}(t))\ddd x  \ddd t \ddd \tau
\end{split}
\end{equation*}
and then with the help of the H\"{o}lder inequality we obtain
\begin{equation}
\label{Ni5}
\begin{split}
&\int_{0}^{T-h}\int_0^h \int_{\mO}  2\bE(\dot{\bu}_0)\cdot \Delta_t^{\tau} (\dot{\bsig}_0-\dot{\bsig})\ddd x \ddd \tau \ddd t\\
&\le C\int_0^h \tau \|\dot{\bsig}_0-\dot{\bsig}\|_{L^{\infty}(0,T;L^2)} \|\bE(\ddot{\bu}_0)\|_{L^1(0,T; L^2)} \ddd \tau \\
&\quad +C\int_0^h  \tau \|\dot{\bsig}_0-\dot{\bsig}\|_{L^{\infty}(0,T;L^2)} \|\bE(\dot{\bu}_0)\|_{L^{\infty}(0,T;L^2)}\ddd \tau\\
&\le Ch^2,
\end{split}
\end{equation}
where we used the assumptions on $\bu_0$ and $\bsig_0$ and the a~priori estimate \eqref{infty1}. The term involving the matrices $\bA$ and $\bH$ on the right hand side of \eqref{Ni3} is estimated as follows
\begin{equation}
\label{Ni6}
\begin{split}
&\int_{0}^{T-h}\int_0^h \int_{\mO}  \Delta_t^{\tau} \left(\bA \dot{\bsig}\cdot \dot{\bsig}+\bH \dot{\bxi}\cdot \dot{\bxi}\right)\ddd x \ddd \tau \ddd t\\
&=\int_0^h \int_{0}^{T-h} \int_{\mO}  \left(\bA \dot{\bsig}\cdot \dot{\bsig}+\bH \dot{\bxi}\cdot \dot{\bxi}\right)(t+\tau)-\left(\bA \dot{\bsig}\cdot \dot{\bsig}+\bH \dot{\bxi}\cdot \dot{\bxi}\right)(t) \ddd x \ddd t \ddd \tau \\
&=\int_0^h \int_{T-h}^{T-h+\tau} \int_{\mO}  \left(\bA \dot{\bsig}\cdot \dot{\bsig}+\bH \dot{\bxi}\cdot \dot{\bxi}\right) \ddd x  \ddd t \ddd \tau\\
&\quad -\int_0^h \int_{0}^{\tau} \int_{\mO}  \left(\bA \dot{\bsig}\cdot \dot{\bsig}+\bH \dot{\bxi}\cdot \dot{\bxi}\right) \ddd x  \ddd t \ddd \tau\\
&\le Ch^2 (\|\dot{\bsig}\|_{L^{\infty}(0,T; L^2)}+\|\dot{\bxi}\|_{L^{\infty}(0,T; L^2)}) \le Ch^2,
\end{split}
\end{equation}
where the estimate \eqref{infty1} is used again. %Substituting \eqref{Ni4}--\eqref{Ni5} into~\eqref{Ni3}, we conclude
%\begin{equation}
%\label{Ni7}
%\begin{split}
%&C_1\int_{t_1}^{t_2}\int_0^h\int_{\mO} C_1|\Delta_t^{\tau}\dot{\bsig}|^2 \ddd \tau \ddd t\\
%&\quad - \int_{t_1}^{t_2}\int_{\mO} 2\mu^{-1}(|\bsig_D|-1)_+ \frac{\bsig_D}{|\bsig_D|} \cdot \left(\int_0^h \Delta_t^{\tau}\dot{\bsig}\ddd \tau\right)\ddd x \ddd t \le Ch^2.
%\end{split}
%\end{equation}
Thus, it remains to evaluate the first term on the right hand side of \eqref{Ni3}. To simplify formula, we set $\bbeta:=\bsig_D-\bxi_D$ and using the convexity of $(|\bbeta|-1)_+^2$, we continue as follows
\begin{equation}
\begin{split}
&\int_{0}^{T-h}\int_{\mO} 2\mu^{-1}(|\bsig_D-\bxi_D|-1)_+ \frac{\bsig_D-\bxi_D}{|\bsig_D-\bxi_D|} \cdot \left(\int_0^h \Delta_t^{\tau}(\dot{\bsig}_D-\dot{\bxi}_D)\ddd \tau\right)\ddd x \ddd t \\
&=\int_{0}^{T-h}\int_{\mO} 2\mu^{-1}(|\bbeta|-1)_+ \frac{\bbeta}{|\bbeta|} \cdot \left(\int_0^h \Delta_t^{\tau}\dot{\bbeta}\ddd \tau\right)\ddd x \ddd t \\
&\quad =\int_{0}^{T-h}\int_{\mO} 2\mu^{-1}(|\bbeta|-1)_+ \frac{\bbeta}{|\bbeta|} \cdot \left({\bbeta}(t+h)-\bbeta(t) -h\dot{\bbeta}(t)\right)\ddd x \ddd t\\
&\quad =\int_{0}^{T-h}\int_{\mO} 2\mu^{-1}(|\bbeta|-1)_+ \frac{\bbeta}{|\bbeta|} \cdot \left({\bbeta}(t+h)-\bbeta(t) -h\dot{\bbeta}(t)\right)\ddd x \ddd t\\
&\quad \le \int_{0}^{T-h}\int_{\mO} \mu^{-1} (|\bbeta(t+h)|-1)^2_+ -(|\bbeta(t)|-1)^2_+ -h\mu^{-1}\frac{d}{dt} (|\bbeta(t)|-1)^2_+ \ddd t\\
&\quad \le \int_{0}^{h}\int_{\mO} \mu^{-1} (|\bbeta(T-h+t)|-1)^2_+ -(|\bbeta(T-h)|-1)^2_+\ddd x \ddd t \\
&\quad \; -\int_0^h\int_{\mO} \mu^{-1} (|\bbeta(t)|-1)^2_+-(|\bbeta(0)|-1)^2_+ \ddd x \ddd t.
\end{split}\label{dont}
\end{equation}
Note that we also used the fact that $|\bbeta(0)|\le 1$, which follows from the fact that $\bxi(0)=0$ and \eqref{SL}. To estimate the right hand side, we notice that \eqref{vztha4} leads to the identity
\begin{equation}
\label{vztha4hl}
\begin{split}
&\frac12 \frac{\ddd}{\ddd t}\int_{\mO} \mu^{-1}(|\bbeta|-1)^2_+ \ddd x \\
&= \int_{\mO}\bH \dot{\bxi} \cdot \dot{\bsig}_{0D}+ (\bE (\dot \bu_0) \cdot (\dot\bsig - \dot\bsig_0) + \bA \dot{\bsig} \cdot \dot{\bsig}_0-\bA \dot{\bsig} \cdot \dot{\bsig}- \bH \dot{\bxi} \cdot \dot{\bxi} \ddd x,
\end{split}
\end{equation}
which after integration over arbitrary interval $(\tau,\tau+\alpha)\subset (0,T)$ and with the help of the a~priori estimate \eqref{infty1} and the assumption on data $\bsig_0$ and $\bu_0$, leads to
\begin{equation}
\label{Ni10}
\begin{split}
&\left|\int_{\mO}\mu^{-1}(|\bbeta(\alpha+\tau)|-1)^2_+ -\mu^{-1}(|\bbeta(\tau)|-1)^2_+\ddd x\right| \\
&\le C\int_{\tau}^{\alpha+\tau} \|\bE (\dot \bu_0)\|_2^2 + \|\dot{\bsig}\|_2^2 + \|\dot{\bsig}_0\|_2^2 + \|\dot{\bxi}\|_2^2 \ddd t \le C \alpha.
\end{split}
\end{equation}
Thus, using this estimate in \eqref{dont}, we see that
\begin{equation}
\begin{split}
&\int_{0}^{T-h}\int_{\mO} 2\mu^{-1}(|\bbeta|-1)_+ \frac{\bbeta}{|\bbeta|} \cdot \left(\int_0^h \Delta_t^{\tau}\dot{\bbeta}\ddd \tau\right)\ddd x \ddd t \\
&\quad \le C\int_{0}^{h}t\ddd t \le Ch^2.
\end{split}\label{dont2}
\end{equation}
Finally, we substitute the estimates \eqref{Ni4}, \eqref{Ni5}, \eqref{Ni6} and \eqref{dont2} into \eqref{Ni3} and finish this part with the uniform estimate
\begin{equation}
\label{Ni11}
\begin{split}
&\frac{1}{h^2}\int_{0}^{T-h}\int_0^h\|\Delta_t^{\tau}\dot{\bsig}\|^2_2+\|\Delta_t^{\tau}\dot{\bxi}\|^2_2 \ddd \tau \ddd t\le C.
\end{split}
\end{equation}
Finally, up to small differences we mimic the procedure from  \cite{FrSc15}, to deduce the proper estimate from \eqref{Ni11}. Indeed, we can compute
\begin{equation*}
\begin{split}
&h^{-1}\int_0^{T-2h}\|\Delta_t^{h}\bbeta\|^2_2  \ddd t\\
&=h^{-1}\int_0^{T-2h}\left\|\frac{1}{h}\int_0^h \bbeta(t+h)-\bbeta(t+h-\tau) +\bbeta(t+h-\tau) -\bbeta(t) \ddd \tau\right\|^2_2   \ddd t\\
&\le 2h^{-2}\int_0^{T-2h}\int_0^h \left\|\bbeta(t+h)-\bbeta(t+h-\tau)\right\|_2^2 +\left\|\bbeta(t+h-\tau) -\bbeta(t) \right\|^2_2 \ddd \tau \ddd t\\
&\le 4h^{-2}\int_0^{T-h}\int_0^h \left\|\Delta_t^{\tau}\bbeta\right\|_2^2\ddd \tau \ddd t.
\end{split}
\end{equation*}
Consequently, it then follows from \eqref{Ni11} that
\begin{equation}
\label{Ni-fin}
\begin{split}
&h^{-1}\int_0^{T-2h}\|\dot{\bsig}(t+h)-\dot{\bsig}(t)\|^2_2 +\|\dot{\bxi}(t+h)-\dot{\bxi}(t)\|^2_2  \ddd t \le C.
\end{split}
\end{equation}

Note that there are only minor changes in the proof for isotropic hardening and therefore we do not provide it here.

\subsection{Fractional spatial regularity for $\dot{\bsig}$ and $\dot{\bxi}$}\label{ss4.1}

The second main result is the spatial fractional regularity upto the boundary in tangential direction and also the interior fractional spatial regularity. It is again based of \cite{FrSp12}, but we do provide here the estimates independent of $\mu$ and extend them up to the boundary when dealing with tangential direction. We keep the notation from the previous section and focus only on the flat boundary. Next, we introduce the space-time shift as follows.
For arbitrary $w$, we denote  $\Delta^{\tau,h}_{t,j} w(t,x):=w(t+\tau,x+he_j)-w(t,x)$, where $e_j$ is the unite vector in the $j$-th direction and $j=1,\ldots, d-1$. We take the scalar product of the first equation in \eqref{vztha} with $-\Delta_{t,j}^{\tau,h} (\dot{\bsig}-\dot{\bsig}_0)\phi^2$, of the second equation in \eqref{vztha} with $-\Delta_{t,j}^{\tau,h} \dot{\bxi}\phi^2$, sum the resulting equalities and finally integrate the result over $\mO$ to get
\begin{equation*}
\begin{split}
&-\int_{\mO} \bA\dot{\bsig}\cdot\Delta_{t,j}^{\tau,h}(\dot{\bsig}-\dot{\bsig}_0)\phi^2 + \bH \dot{\bxi} \cdot \Delta_{t,j}^{\tau,h}\dot{\bxi}\phi^2 \ddd x\\
&\qquad -\int_{\mO} \mu^{-1}(|\bsig_D-\bxi_D|-1)_+ \frac{\bsig_D-\bxi_D}{|\bsig_D-\bxi_D|} \cdot \Delta_{t,j}^{\tau,h} (\dot{\bsig}_D-\dot{\bxi}_D)\phi^2\ddd x \\
&= \int_{\mO} \bE (\dot{\bu}_0-\dot{\bu})\cdot \Delta_{t,j}^{\tau,h} (\dot{\bsig}-\dot{\bsig}_0)\phi^2 - \bE(\dot{\bu}_0)\cdot \Delta_{t,j}^{\tau,h} (\dot{\bsig}-\dot{\bsig}_0)\phi^2- \bH \dot{\bxi} \cdot\Delta_{t,j}^{\tau,h}\dot{\bsig}_{0D}\phi^2\ddd x.
\end{split}
\end{equation*}
Next, we focus on the first term on the left hand side. Similarly as before, we have
$$
\begin{aligned}
-\bA\dot{\bsig}\cdot \Delta_{t,j}^{\tau,h}\dot{\bsig}&=\frac12 \bA\Delta_{t,j}^{\tau,h}\dot{\bsig}\cdot \Delta_{t,j}^{\tau,h} \dot{\bsig}-\frac12 \Delta_{t,j}^{\tau,h}\left(\bA \dot{\bsig}\cdot \dot{\bsig}\right),\\
-\bH\dot{\bxi}\cdot \Delta_{t,j}^{\tau,h} \dot{\bxi}&=\frac12 \bH\Delta_{t,j}^{\tau,h}\dot{\bxi}\cdot \Delta_{t,j}^{\tau,h} \dot{\bxi}-\frac12 \Delta_{t,j}^{\tau,h} \left(\bH \dot{\bxi}\cdot \dot{\bxi}\right).
\end{aligned}
$$
Therefore, using this identities and also the ellipticity condition \eqref{ellip}, we get after integration with respect to $t\in (0,T-h)$ and $\tau\in (0,h)$
\begin{equation}
\label{Ni2s}
\begin{split}
&C_1\int_0^{T-h}\int_0^h \|\Delta_{t,j}^{\tau,h}\dot{\bsig}\phi\|_2^2+\|\Delta_{t,j}^{\tau,h}\dot{\bxi}\phi\|_2^2\ddd \tau \ddd t\\
&\le 2\int_0^{T-h}\int_{\mO} \mu^{-1}(|\bsig_D-\bxi_D|-1)_+ \frac{\bsig_D-\bxi_D}{|\bsig_D-\bxi_D|} \times\\
 &\qquad {} \qquad \times \left(\int_0^h\Delta_{t,j}^{\tau,h} (\dot{\bsig}_D-\dot{\bxi}_D)\ddd \tau \right)\phi^2\ddd x \ddd t\\
&+\int_0^{T-h}\int_0^h\int_{\mO} \Delta_{t,j}^{\tau,h}(\bA \dot{\bsig}\cdot \dot{\bsig})\phi^2+\Delta_{t,j}^{\tau,h} (\bH \dot{\bxi}\cdot \dot{\bxi})\phi^2\ddd x \ddd \tau \ddd t\\
&+ 2\int_0^{T-h}\int_0^h\int_{\mO} \bE (\dot{\bu}_0-\dot{\bu})\cdot \Delta_{t,j}^{\tau,h} (\dot{\bsig}-\dot{\bsig}_0)\phi^2 - \bE(\dot{\bu}_0)\cdot \Delta_{t,j}^{\tau,h} \dot{\bsig}\phi^2 \ddd x \ddd \tau \ddd t\\
 &+2\int_0^{T-h}\int_0^h\int_{\mO} (\bE(\dot{\bu}_0)- \bH \dot{\bxi}-\bA \dot{\bsig})\cdot \Delta_{t,j}^{\tau,h} \dot{\bsig}_0\phi^2 \ddd x \ddd \tau \ddd t.
\end{split}
\end{equation}
%
%%
%%{\tt END}
%%
%%The first term on the right hand side vanishes and after the use of the H\"{o}lder inequality and reorganisation of all terms and the bound \eqref{ellip}, we deduce that
%%\begin{equation}
%%\label{Ni1}
%%\begin{split}
%%&-\int_{\mO} \bA\dot{\bsig}\cdot\Delta_t^{\tau} \dot{\bsig} + \bH \dot{\bxi} \cdot \Delta_t^\tau \dot{\bxi} \ddd x\\
%%&\qquad -\int_{\mO} \mu^{-1}(|\bsig_D-\bxi_D|-1)_+ \frac{\bsig_D-\bxi_D}{|\bsig_D-\bxi_D|} \cdot \Delta_t^{\tau} (\dot{\bsig}_D-\dot{\bxi}_D)\ddd x \\
%%&\le C(\|\dot{\bsig}\|_2 + \|\dot{\bxi}\|_2) \|\Delta_t^{\tau}\dot{\bsig}_{0}\|_2 + \int_{\mO} \bE(\dot{\bu}_0)\cdot \Delta_t^{\tau} (\dot{\bsig}_0-\dot{\bsig})\ddd x.
%%\end{split}
%%\end{equation}
%%
%%Next, we integrate the resulting inequality with respect to $\tau$ over the interval $(0,h)$ and with respect to $t$ over the interval $(0,T-h)$ to get
%%\begin{equation}
%%\label{Ni3}
%%\begin{split}
%%&C_1\int_{0}^{T-h}\int_0^h \|\Delta_t^{\tau}\dot{\bsig}\|^2_2 +\|\Delta_t^{\tau}\dot{\bxi}\|^2_2  \ddd \tau \ddd t \\
%%&\le \int_{0}^{T-h}\int_{\mO} 2\mu^{-1}(|\bsig_D-\bxi_D|-1)_+ \frac{\bsig_D-\bxi_D}{|\bsig_D-\bxi_D|} \cdot \left(\int_0^h \Delta_t^{\tau}(\dot{\bsig}_D-\dot{\bxi}_D)\ddd \tau\right)\ddd x \ddd t \\
%%&\quad+ \int_{0}^{T-h}\int_0^h \int_{\mO}  2\bE(\dot{\bu}_0)\cdot \Delta_t^{\tau} (\dot{\bsig}_0-\dot{\bsig})+\Delta_t^{\tau} \left(\bA \dot{\bsig}\cdot \dot{\bsig}+\bH \dot{\bxi}\cdot \dot{\bxi}\right)\ddd x\ddd \tau \ddd t\\
%%&\qquad  +C\int_{0}^{T-h}\int_0^h \|\Delta_t^{\tau} \dot{\bsig}_0\|_{2}\ddd \tau \ddd t.
%%\end{split}
%%\end{equation}
%
Our goal is to provide uniform estimates for all terms on the right hand side.

\bigskip

\paragraph{\bf The fourth term in \eqref{Ni2s}}
We start with the last term. Using the H\"{o}lder inequality and the characterization of Sobolev spaces as well as the uniform bound~\eqref{infty1} and the assumption on $\bu_0$, we have
\begin{equation}
\label{Ni4s}
\begin{split}
&\int_0^{T-h}\int_0^h\int_{\mO} (\bE(\dot{\bu}_0)- \bH \dot{\bxi}-\bA \dot{\bsig})\cdot \Delta_{t,j}^{\tau,h} \dot{\bsig}_0\phi^2 \ddd x \ddd \tau \ddd t\\
&\le C\int_0^{T-h}\int_0^h (\|\bE(\dot{\bu}_0)\|_2 + \|\dot{\bxi}\|_2 +\|\dot{\bsig}\|) \|\Delta_{t,j}^{\tau,h} \dot{\bsig}_0\phi^2\|_2 \ddd \tau \ddd t\\
&\le C\int_0^{T-h}\int_0^h  \|(\Delta_{t,j}^{\tau,h} \dot{\bsig}_0-\Delta_{j}^{h} \dot{\bsig}_0)\phi^2\|_2+\|\Delta_{j}^{h} \dot{\bsig}_0\phi^2\|_2 \ddd \tau \ddd t\\
&\le Ch\int_0^h \int_0^{T-h}\|\ddot{\bsig}_0\|_2+ \|\nabla \dot{\bsig}_0\|_2 \ddd \tau \ddd t \le Ch^2.
\end{split}
\end{equation}

\bigskip

\paragraph{\bf The third term in \eqref{Ni2s}}
In the third term in \eqref{Ni2s}, we first split the time-space shift to the space shift and the time shift. Next, for the time shift, we again use the equation \eqref{vztha}, while for the space shift we use the integration by parts (note that the boundary term vanishes since we have shifts only in tangential direction and we also know that $\diver (\bsig - \bsig_0)=0$) and then we also move the shift from $\bsig$ to other terms. In addition, we keep all shifts on $\bsig_0$ since it has sufficient regularity. More precisely, we have
\begin{equation*}
\begin{split}
&\int_0^{T-h}\int_0^h\int_{\mO} \bE (\dot{\bu}_0-\dot{\bu})\cdot \Delta_{t,j}^{\tau,h} (\dot{\bsig}-\dot{\bsig}_0)\phi^2 - \bE(\dot{\bu}_0)\cdot \Delta_{t,j}^{\tau,h} \dot{\bsig}\phi^2 \ddd x \ddd \tau \ddd t\\
&=\int_0^{T-h}\int_0^h\int_{\mO} \bE (\dot{\bu}_0-\dot{\bu})\cdot \Delta_{t}^{\tau} (\dot{\bsig}-\dot{\bsig}_0)\phi^2 - \bE(\dot{\bu}_0)\cdot \Delta_{t,j}^{\tau,h} \dot{\bsig}\phi^2 \ddd x \ddd \tau \ddd t\\
&\quad +\int_0^{T-h}\int_0^h\int_{\mO} \bE (\dot{\bu}_0-\dot{\bu})\cdot (\Delta_{t,j}^{\tau,h} (\dot{\bsig}-\dot{\bsig}_0)-\Delta_{t}^{\tau} (\dot{\bsig}-\dot{\bsig}_0))\phi^2 \ddd x \ddd \tau \ddd t\\
&=\int_0^{T-h}\int_0^h\int_{\mO} \bE (\dot{\bu}_0)\cdot (\Delta_{t}^{\tau} \dot{\bsig}-\Delta_{t,j}^{\tau,h} \dot{\bsig})\phi^2 \ddd x \ddd \tau \ddd t\\
&\quad \int_0^{T-h}\int_0^h\int_{\mO}(\bE (\dot{\bu})-\bE (\dot{\bu}_0))\cdot \Delta_{t}^{\tau} \dot{\bsig}_0\phi^2 - \bE (\dot{\bu})\cdot \Delta_{t}^{\tau} \dot{\bsig}\phi^2  \ddd x \ddd \tau \ddd t\\
&\quad +\int_0^{T-h}\int_0^h\int_{\mO} \nabla( \phi^2(\dot{\bu}_0-\dot{\bu}))\cdot (\Delta_{t,j}^{\tau,h} (\dot{\bsig}-\dot{\bsig}_0)-\Delta_{t}^{\tau} (\dot{\bsig}-\dot{\bsig}_0)) \ddd x \ddd \tau \ddd t\\
&\quad -\int_0^{T-h}\int_0^h\int_{\mO} ((\dot{\bu}_0-\dot{\bu})\otimes \nabla \phi^2)\cdot (\Delta_{t,j}^{\tau,h} (\dot{\bsig}-\dot{\bsig}_0)-\Delta_{t}^{\tau} (\dot{\bsig}-\dot{\bsig}_0)) \ddd x \ddd \tau \ddd t\\
&=-\int_0^{T-h}\int_0^h\int_{\mO} \Delta_j^{-h}(\bE (\dot{\bu}_0)\phi^2)\cdot \dot{\bsig}(t+\tau)\ddd x \ddd \tau \ddd t\\
&\quad +\int_0^{T-h}\int_0^h\int_{\mO}(\bE (\dot{\bu})-\bE (\dot{\bu}_0))\cdot \Delta_{t}^{\tau} \dot{\bsig}_0\phi^2 \ddd x \ddd \tau \ddd t\\
&\quad -\int_0^{T-h}\int_0^h\int_{\mO} \bA \dot{\bsig} \cdot \Delta_{t}^{\tau} \dot{\bsig}\phi^2 +\bH \dot{\bxi} \cdot \Delta_{t}^{\tau} \dot{\bxi}\phi^2 \ddd x + \bH \dot{\bxi} \cdot \Delta_{t}^{\tau} (\dot{\bsig}_D-\dot{\bxi}_D)\phi^2 \ddd \tau \ddd t\\
&\quad -\int_0^{T-h}\int_0^h\int_{\mO} \Delta_j^h((\dot{\bu}_0-\dot{\bu})\otimes \nabla \phi^2)\cdot ((\dot{\bsig}(t+\tau)-\dot{\bsig}_0(t+\tau))) \ddd x \ddd \tau \ddd t.
\end{split}
\end{equation*}
Next, we apply the H\"{o}lder inequality, use the symmetry of $\bA$ and $\bH$, the assumption on $\bu_0$ and $\bsig_0$, the uniform bounds \eqref{infty1}--\eqref{infty2} and the time regularity estimate~\eqref{Ni11} to conclude
\begin{equation}
\begin{split}
\label{Ni5s}
&\int_0^{T-h}\int_0^h\int_{\mO} \bE (\dot{\bu}_0-\dot{\bu})\cdot \Delta_{t,j}^{\tau,h} (\dot{\bsig}-\dot{\bsig}_0)\phi^2 - \bE(\dot{\bu}_0)\cdot \Delta_{t,j}^{\tau,h} \dot{\bsig}\phi^2 \ddd x \ddd \tau \ddd t\\
&\le \int_0^{T-h}\int_0^h h\|\bE (\dot{\bu}_0)\phi^2\|_{1,2} \|\dot{\bsig}(t+\tau)\|_2+\tau \|\bE (\dot{\bu})-\bE (\dot{\bu}_0)\|_2 \|\ddot{\bsig}_0\|_2 \ddd \tau \ddd t\\
&\quad +\frac12 \int_0^{T-h}\int_0^h\int_{\mO}\bA \Delta_{t}^{\tau} \dot{\bsig} \cdot \Delta_{t}^{\tau} \dot{\bsig}\phi^2 +\bH \Delta_{t}^{\tau} \dot{\bxi} \cdot \Delta_{t}^{\tau} \dot{\bxi}\phi^2 \ddd x \ddd \tau \ddd t \\
 &\quad -\int_0^{T-h}\int_0^h\int_{\mO}\bH \dot{\bxi} \cdot \Delta_{t}^{\tau} (\dot{\bsig}_D-\dot{\bxi}_D)\phi^2 \ddd x \ddd \tau \ddd t\\
&\quad -\frac12 \int_0^{T-h}\int_0^h\int_{\mO} \Delta_{t}^{\tau}(\bA \dot{\bsig} \cdot  \dot{\bsig})\phi^2 +\Delta_{t}^{\tau}(\bH \dot{\bxi} \cdot \dot{\bxi})\phi^2 \ddd x  \ddd \tau \ddd t\\
&\quad +\int_0^{T-h}\int_0^h h\|(\dot{\bu}_0-\dot{\bu})\otimes \nabla \phi^2\|_{1,2} \|\dot{\bsig}(t+\tau)-\dot{\bsig}_0(t+\tau)\|_2  \ddd \tau \ddd t\\
&\le Ch^2 -\frac12\int_0^{T-h}\int_0^h\int_{\mO} \Delta_{t}^{\tau}(\bA \dot{\bsig} \cdot  \dot{\bsig})\phi^2 +\Delta_{t}^{\tau}(\bH \dot{\bxi} \cdot \dot{\bxi})\phi^2 \ddd x  \ddd \tau \ddd t\\
&\quad  +\int_0^{T-h}\int_0^h\int_{\mO} \bH \dot{\bxi} \cdot \Delta_{t}^{\tau} (\dot{\bsig}_D-\dot{\bxi}_D)\phi^2 \ddd x  \ddd \tau \ddd t.
\end{split}
\end{equation}
There are still remaining two terms on the right hand side. But for the first one we can use exactly the same computation as in \eqref{Ni6} and observe
\begin{equation}\label{alreadyd}
\left| \int_0^{T-h}\int_0^h\int_{\mO} \Delta_{t}^{\tau}(\bA \dot{\bsig} \cdot  \dot{\bsig})\phi^2 +\Delta_{t}^{\tau}(\bH \dot{\bxi} \cdot \dot{\bxi})\phi^2 \right| \le Ch^2.
\end{equation}
To estimate the remaining term in \eqref{Ni5s}, we again use the abbreviation $\bbeta:=\bsig_D-\bxi_D$ and following the computation in \eqref{dont} we deduce
\begin{equation}
\begin{split}
&2\int_0^{T-h}\int_0^h\int_{\mO}\bH \dot{\bxi} \cdot \Delta_{t}^{\tau} (\dot{\bsig}_D-\dot{\bxi}_D)\phi^2 \ddd x \ddd \tau \ddd t\\
&=\int_{0}^{T-h}\int_{\mO} 2\mu^{-1}(|\bbeta|-1)_+ \frac{\bbeta}{|\bbeta|} \cdot \left(\int_0^h \Delta_t^{\tau}\dot{\bbeta}\ddd \tau\right)\phi^2\ddd x \ddd t \\
&\quad \le \int_{0}^{h}\int_{\mO} \mu^{-1} (|\bbeta(T-h+t)|-1)^2_+\phi^2 -(|\bbeta(T-h)|-1)^2_+\phi^2\ddd x \ddd t \\
&\quad \; -\int_0^h\int_{\mO} \mu^{-1} (|\bbeta(t)|-1)^2_+\phi^2 -(|\bbeta(0)|-1)^2_+\phi^2 \ddd x \ddd t.
\end{split}\label{donts}
\end{equation}
Similarly as before (compare with \eqref{vztha4hl}), we also have the identity
\begin{equation}
\label{vztha4hls}
\begin{split}
&\frac12 \frac{\ddd}{\ddd t}\int_{\mO} \mu^{-1}(|\bbeta|-1)_+ \phi^2 \ddd x \\
&= \int_{\mO}(\bE (\dot \bu) \cdot \dot\bsig \phi^2 -\bA \dot{\bsig} \cdot \dot{\bsig}\phi^2- \bH \dot{\bxi} \cdot \dot{\bxi}\phi^2 \ddd x,
\end{split}
\end{equation}
which after integration over arbitrary interval $(\tau,\tau+\alpha)\subset (0,T)$ and with the help of the a~priori estimate \eqref{infty1} and \eqref{infty2}, leads to
\begin{equation}
\label{Ni10s}
\begin{split}
&\left|\int_{\mO}\mu^{-1}(|\bbeta(\alpha+\tau)|-1)^2_+ -\mu^{-1}(|\bbeta(\tau)|-1)^2_+\ddd x\right| \le C \alpha.
\end{split}
\end{equation}
Thus, using this estimate in \eqref{donts}, we see that
\begin{equation}
\begin{split}
&\int_0^{T-h}\int_0^h\int_{\mO}\bH \dot{\bxi} \cdot \Delta_{t}^{\tau} (\dot{\bsig}_D-\dot{\bxi}_D)\phi^2 \ddd x \ddd \tau \ddd t \le Ch^2.
\end{split}\label{dont2s}
\end{equation}

\bigskip

\paragraph{\bf The second term in \eqref{Ni2s}}
Again here, we split the time and space shift and then use \eqref{alreadyd} and the spatial regularity of $\phi$ as follows
\begin{equation}
\begin{split}\label{lips}
&\int_0^{T-h}\int_0^h\int_{\mO} \Delta_{t,j}^{\tau,h}(\bA \dot{\bsig}\cdot \dot{\bsig})\phi^2+\Delta_{t,j}^{\tau,h} (\bH \dot{\bxi}\cdot \dot{\bxi})\phi^2\ddd x \ddd \tau \ddd t\\
&=\int_0^{T-h}\int_0^h\int_{\mO} (\Delta_{t,j}^{\tau,h}(\bA \dot{\bsig}\cdot \dot{\bsig}+\bH \dot{\bxi}\cdot \dot{\bxi})-\Delta_{t}^{\tau}(\bA \dot{\bsig}\cdot \dot{\bsig}+\bH \dot{\bxi}\cdot \dot{\bxi}))\phi^2\ddd x \ddd \tau \ddd t\\
&\quad +\int_0^{T-h}\int_0^h\int_{\mO}\Delta_{t}^{\tau}(\bA \dot{\bsig}\cdot \dot{\bsig}+\bH \dot{\bxi}\cdot \dot{\bxi})\phi^2\ddd x \ddd \tau \ddd t\\
&=-\int_0^{T-h}\int_0^h\int_{\mO} (\bA \dot{\bsig}\cdot \dot{\bsig}+\bH \dot{\bxi}\cdot \dot{\bxi})(t+\tau) (\Delta_{j}^{-h}\phi^2)\ddd x \ddd \tau \ddd t\\
&\quad +\int_0^{T-h}\int_0^h\int_{\mO}\Delta_{t}^{\tau}(\bA \dot{\bsig}\cdot \dot{\bsig}+\bH \dot{\bxi}\cdot \dot{\bxi})\phi^2\ddd x \ddd \tau \ddd t\\
&\le Ch^2 + Ch\int_0^{T}\int_0^h \|\dot{\bsig}\|_2^2 + \|\dot{\bxi}\|_2^2\ddd \tau \ddd t\le Ch^2.
\end{split}
\end{equation}

\bigskip

\paragraph{\bf The first term in \eqref{Ni2s}}
We again use the abbreviation $\bbeta:= \bsig_D-\bxi_D$. The we estimate the first term in \eqref{Ni2s} very similarly as in \eqref{dont} as follows. First, we rewrite it in the following way
\begin{equation*}
\begin{split}%\label{dontx}
&2\int_0^{T-h}\int_{\mO} \mu^{-1}(|\bbeta|-1)_+ \frac{\bbeta}{|\bbeta|} \cdot \left(\int_0^h\Delta_{t,j}^{\tau,h} \dot{\bbeta}\ddd \tau \right)\phi^2\ddd x \ddd t\\
&=2\int_0^{T-h}\int_{\mO} \mu^{-1}(|\bbeta|-1)_+ \frac{\bbeta}{|\bbeta|} \cdot \left(\Delta_{t,j}^{h,h}\bbeta - \Delta_{j}^h\bbeta-h \dot{\bbeta} \right)\phi^2\ddd x \ddd t\\
&=2\int_0^{T-h}\int_{\mO} \mu^{-1}(|\bbeta|-1)_+ \frac{\bbeta}{|\bbeta|} \cdot \Delta_{t,j}^{h,h}\bbeta \phi^2\ddd x \ddd t\\
&\quad -2\int_0^{T-h}\int_{\mO} \left(\mu^{-1}(|\bbeta|-1)_+ \frac{\bbeta}{|\bbeta|}\right)(t,x+he_j) \cdot  \Delta_{j}^h\bbeta\phi^2\ddd x \ddd t\\
&\quad -h\int_0^{T-h}\frac{\ddd}{\ddd t}\int_{\mO} \mu^{-1}(|\bbeta|-1)^2_+ \phi^2\ddd x \ddd t\\
&\quad+2\int_0^{T-h}\int_{\mO} \Delta_j^h \left(\mu^{-1}(|\bbeta|-1)_+ \frac{\bbeta}{|\bbeta|} \right)\cdot  \Delta_{j}^h\bbeta\phi^2\ddd x \ddd t
\end{split}
\end{equation*}
and then using the convexity and the fact $|\bbeta(0)|\le 1$, we can estimate it as
\begin{equation}
\begin{split}\label{dontx}
&2\int_0^{T-h}\int_{\mO} \mu^{-1}(|\bbeta|-1)_+ \frac{\bbeta}{|\bbeta|} \cdot \left(\int_0^h\Delta_{t,j}^{\tau,h} \dot{\bbeta}\ddd \tau \right)\phi^2\ddd x \ddd t\\
&\le \int_0^{T-h}\int_{\mO} (\Delta_{t,j}^{h,h} (\mu^{-1}(|\bbeta|-1)^2_+)- \Delta_{j}^{h} (\mu^{-1}(|\bbeta|-1)^2_+)) \phi^2\ddd x \ddd t\\
&\quad -h\int_{\mO} (\mu^{-1}(|\bbeta(T-h)|-1)^2_+ - \mu^{-1}(|\bbeta(0)|-1)^2_+) \phi^2\ddd x \\
&\quad+2\int_0^{T-h}\int_{\mO} \Delta_j^h \left(\mu^{-1}(|\bbeta|-1)_+ \frac{\bbeta}{|\bbeta|} \right)\cdot  \Delta_{j}^h\bbeta\phi^2\ddd x \ddd t\\
&=\int_0^{T-h}\int_{\mO} \Delta_{t}^{h} (\mu^{-1}(|\bbeta|-1)^2_+))\phi^2(x-he_j)\ddd x \ddd t\\
&\quad -h\int_{\mO} (\mu^{-1}(|\bbeta(T-h)|-1)^2_+ - \mu^{-1}(|\bbeta(0)|-1)^2_+) \phi^2(x-he_j)\ddd x \\
&\quad +h\int_{\mO} (\mu^{-1}(|\bbeta(T-h)|-1)^2_+ - \mu^{-1}(|\bbeta(0)|-1)^2_+) \Delta^{-h}_j\phi^2\ddd x \\
&\quad+2\int_0^{T-h}\int_{\mO} \Delta_j^h \left(\mu^{-1}(|\bbeta|-1)_+ \frac{\bbeta}{|\bbeta|} \right)\cdot  \Delta_{j}^h\bbeta\phi^2\ddd x \ddd t.
\end{split}
\end{equation}
Next, we estimate the remaining terms separately. First, using the very similar procedure as in \eqref{dont} and consequent inequalities, we observe
\begin{equation*}
\begin{split}
&\int_0^{T-h}\int_{\mO} \Delta_{t}^{h} (\mu^{-1}(|\bbeta|-1)^2_+))\phi^2(x-he_j)\ddd x \ddd t\\
&\quad -h\int_{\mO} (\mu^{-1}(|\bbeta(T-h)|-1)^2_+ - \mu^{-1}(|\bbeta(0)|-1)^2_+) \phi^2(x-he_j)\ddd x \le Ch^2,
\end{split}
\end{equation*}
where the constant $C$ depends on data and on $\phi$. Next, by uniform bound \eqref{Ni10}, we also have the estimate
\begin{equation*}
\begin{split}
&h\int_{\mO} (\mu^{-1}(|\bbeta(T-h)|-1)^2_+ - \mu^{-1}(|\bbeta(0)|-1)^2_+) \Delta^{-h}_j\phi^2\ddd x \\
&\le C h^2 \|\nabla \phi^2\|_{\infty} \int_{\mO} \mu^{-1}(|\bbeta(T-h)|-1)^2_+ \ddd x \le Ch^2.
\end{split}
\end{equation*}
Finally, for the remaining term, we apply the $\Delta_j^h$ to \eqref{vztha} and test by $\Delta_j^h \bsig$ and by  $\Delta_j^h \bxi$. Then we can repeat the same procedure as in estimating the tangential derivatives for $\bsig$ and $\bxi$ and deduce again
\begin{equation*}
\begin{split}
&2\int_0^{T-h}\int_{\mO} \Delta_j^h \left(\mu^{-1}(|\bbeta|-1)_+ \frac{\bbeta}{|\bbeta|} \right)\cdot  \Delta_{j}^h\bbeta\phi^2\ddd x \ddd t\le Ch^2.
\end{split}
\end{equation*}

\bigskip

\paragraph{\bf Summarizing the estimates}
Hence, if we use the above estimates in \eqref{Ni2s}, we get
\begin{equation}
\label{Ni11s}
\begin{split}
&\frac{1}{h^2}\int_{0}^{T-h}\int_0^h\|\phi\Delta_{t,j}^{\tau,h}\dot{\bsig}\|^2_2+\|\phi\Delta_{t,j}^{\tau,h}\dot{\bxi}\|^2_2 \ddd \tau \ddd t\le C
\end{split}
\end{equation}
for arbitrary $j=1,\ldots,d-1$ and localization function $\phi$. Then, by a simple inequality, we deduce
\begin{equation}
\label{Ni12s}
\begin{split}
&\frac{1}{h}\int_{0}^{T-h}\|\phi\Delta_{j}^{h}\dot{\bsig}\|^2_2+\|\phi\Delta_{j}^{h}\dot{\bxi}\|^2_2 \ddd t\\
&=\frac{1}{h^2}\int_{0}^{T-h}\int_0^h\|\phi\Delta_{j}^{h}\dot{\bsig}\|^2_2+\|\phi\Delta_{j}^{h}\dot{\bxi}\|^2_2 \ddd \tau \ddd t\\
&\le \frac{C}{h^2}\int_{0}^{T-h}\int_0^h\|\phi(\Delta_{t,j}^{\tau,h}\dot{\bsig}-\Delta_{j}^{h}\dot{\bsig}\|_2^2 +\|\phi\Delta_{t,j}^{\tau,h}\dot{\bsig}\|_2^2 +\|\phi(\Delta_{t,j}^{\tau,h}\dot{\bxi}-\Delta_{j}^{h}\dot{\bxi}\|_2^2 \\
&\qquad {} \qquad +\|\phi\Delta_{t,j}^{\tau,h}\dot{\bxi}\|_2^2 \ddd \tau \ddd t\\
&\le \frac{C}{h^2}\int_{0}^{T-h}\int_0^h\|\Delta_{t}^{\tau}\dot{\bsig}\|_2^2 +\|\phi\Delta_{t,j}^{\tau,h}\dot{\bsig}\|_2^2 +\|\phi(\Delta_{t}^{\tau}\dot{\bxi}\|_2^2 +\|\phi\Delta_{t,j}^{\tau,h}\dot{\bxi}\|_2^2 \ddd \tau \ddd t
\le C,
\end{split}
\end{equation}
where the last inequality follows from time regularity estimates and \eqref{Ni11s}.

\section{Normal derivatives estimates}
\label{normal}
In this final part, we derive the estimate for the normal derivative. Note that the estimate is again uniform with respect to $\mu$. We also keep the notation near the boundary and use the function $\phi$ which is compactly supported in a cube $(-1+h_0,1-h_0)^{d}$ and equal to one in a cube $(-1+2h_0,1-2h_0)^d$. We start this part by using already proven time fractional regularity to transfer also the spatial regularity.

\subsection{Estimate for for $\dot{\bsig}$ and $\dot{\bxi}$ via time interpolation}
Here, we show how the fractional estimates in the $d$-direction for $\bsig$ and $\bxi$ and the fractional estimates in the time direction for $\dot{\bsig}$ and $\dot \bxi$ can improve the spatial regularity of $\dot{\bsig}$ and $\dot{\bxi}$. The key estimate is formulated in the following.
\begin{Lemma}\label{interpol1}
Let $\phi$ be as above. Then for any $\delta\in (0,\frac13)$, the solution satisfies
\begin{equation}
\label{interpol2}
\begin{split}
\int_0^T &\int_{\mO}|\Delta^h_d \dot{\bsig}|^2\phi^2 + |\Delta^h_d \dot{\bxi}|^2\phi^2\ddd x \ddd t \\
 &\le C(\delta)\left(\int_0^T \int_{\mO} |\Delta^h_d {\bsig}|^2\phi^2 + |\Delta^h_d {\bxi}|^2\phi^2\ddd x \ddd t\right)^{\frac13 -\delta},
\end{split}
\end{equation}
where the constant $C(\delta)$ depends only on data and explodes as $\delta\to 0_+$.
\end{Lemma}

\begin{proof}
The proof is based on the interpolation of Bochner-Sobolev spaces. We recall the classical Bochner-Sobolev interpolation between $W^{\alpha,2}$ and $L^2$ and also the Nikolskii-Sobolev embedding $N^{\frac32,2}\hookrightarrow W^{\alpha,2}$ valid for all $\alpha <\frac32$ to get (we use any $\alpha \in (1,\frac32)$)
\begin{equation}\label{sobolev}
\begin{split}
&\int_0^T \|\dot{f}\|_2^2\ddd t \le  C\|f\|_{L^2(0,T;L^2(\mO))}^{2-\frac{2}{\alpha}} \|f\|^{\frac{2}{\alpha}}_{W^{\alpha,2}(0,T; L^2(\mO))} \\
&\le  C\|f\|_{L^2(0,T;L^2(\mO))}^{2-\frac{2}{\alpha}} \|f\|^{\frac{2}{\alpha}}_{N^{\frac32,2}(0,T; L^2(\mO))} \\
 &\le C\int_0^T \|f\|_2^2 \ddd t+ C(\alpha)\left( \sup_{h\in (0,T)}\int_0^{T-h} \frac{\|\Delta_t^h \dot{f}\|_2^2}{h} \ddd t\right)^{\frac{1}{\alpha}}\left(\int_0^T \|f\|_2^2 \ddd t\right)^{1-\frac{1}{\alpha}},
\end{split}
\end{equation}
where the constant $C(\alpha)$ explodes as $\alpha \to \frac32$. Consequently, using the above estimate on $f:=\Delta^h_d \bsig \phi $ and $\Delta^h_d \bxi\phi$ and using the a~priori bound \eqref{Ni-fin}, we see that
\begin{equation}\label{sobolev2}
\begin{split}
\int_0^T &\int_{\mO}|\Delta^h_d \dot{\bsig}|^2\phi^2 + |\Delta^h_d \dot{\bxi}|^2\phi^2\ddd x \ddd t \\
 &\le C\int_0^T \int_{\mO} |\Delta^h_d {\bsig}|^2\phi^2 + |\Delta^h_d {\bxi}|^2\phi^2\ddd x \ddd t \\
 &\qquad + C(\alpha)\left(\int_0^T \int_{\mO} |\Delta^h_d {\bsig}|^2\phi^2 + |\Delta^h_d {\bxi}|^2\phi^2\ddd x \ddd t\right)^{1-\frac{1}{\alpha}}.
\end{split}
\end{equation}
Since $\alpha\in (1,\frac32)$ can be arbitrary and we have the control \eqref{infty1}, the estimate \eqref{interpol2} follows.
\end{proof}

\subsection{First estimate for $\bsig$ and $\bxi$}
Here, we start with an estimate, that directly leads to $\frac12$ regularity of the stress and hardening, but it will also  serve later for the bootstrap argument.
\begin{Lemma}\label{L-n1} Let $\phi$ be chosen such that $\phi(x',s)$ is independent of $s$ for all $s\in [0,h_0]$ with $h_0>0$. Then for all $h\in (0,h_0)$ the following estimate holds
\begin{equation}\label{nor-start}
\begin{split}
\sup_{t\in (0,T)} &\|\phi \Delta_d^h \bsig(t)\|_2^2 + \|\phi \Delta_d^h \bxi(t)\|_2^2 \\
&\le Ch^{\frac32}+ Ch\int_0^T \left(\int_{\mO \cap \{x_d\in (0,h)\}} \!\!\!\!\!\! |D_d\dot{\bu}\phi|^2\ddd x\right)^{\frac12}\ddd t.
\end{split}
\end{equation}
In addition, if the $\supp \phi \cap \overline{\mO}_N=\emptyset$ then we have
\begin{equation}\label{nor-start2}
\begin{split}
\sup_{t\in (0,T)} &\|\phi \Delta_d^h \bsig(t)\|_2^2 + \|\phi \Delta_d^h \bxi(t)\|_2^2 \\
&\le Ch^{\frac32}+Ch\int_0^T\left(\int_{\mO \cap \{x_d\in (0,h)\}} |\bE(\dot{\bu})\phi|^2\ddd x\right)^{\frac12}\ddd t.
\end{split}
\end{equation}
\end{Lemma}
\begin{proof}
We apply the operator $\Delta^h_d$ to both equations in \eqref{vztha} and and test by $\Delta_d^h \bsig \phi^2$ and $\Delta_d^h \bxi \phi^2$ respectively. Note that since $x+he_d\in \mO$ such operation is well defined. Thus, doing so, we observe
\begin{equation}
\label{QQ1}
\begin{split}
&\frac12 \frac{\ddd}{\ddd t} \int_{\mO} \bA (\Delta_d^h \bsig -\Delta_d^h \bsig_0) \cdot (\Delta_d^h \bsig -\Delta_d^h \bsig_0) \phi^2 \ddd x\\
&+\frac12 \frac{\ddd}{\ddd t}  \bH (\Delta_d^h \bxi -\Delta_d^h \bxi_0) \cdot (\Delta_d^h \bxi -\Delta_d^h \bxi_0) \phi^2 \ddd x\\
 &\quad + \int_{\mO} \Delta_d^j \left( \mu^{-1}(|\bsig_D-\bxi_D|-1)_+ \frac{\bsig_D-\bxi_D}{|\bsig_D-\bxi_D|} \right) \cdot \Delta_d^h (\bsig_D -\bxi_D)\phi^2\ddd x\\
&= \int_{\mO} \Delta_d^h\bE(\dot{\bu}-\dot{\bu}_0) \cdot \Delta_d^h (\bsig-\bsig_0) \phi^2\ddd x \\
&\quad +\int_{\mO}\Delta_d^h\bE(\dot{\bu}_0) \cdot \Delta_d^h (\bsig-\bsig_0) \phi^2  - \Delta_d^h(\bA  \dot{\bsig}_0+\bH \dot{\bxi}_0)\cdot (\Delta_d^h \bsig-\Delta_d^h \bsig_0)\phi^2 \ddd x.
\end{split}
\end{equation}
The terms on the left hand side are those from which we read information. The second term on the right hand side can be easily estimated and prepared for the Grownall lemma as follows
\begin{equation}
\label{QQ2}
\begin{split}
&\int_{\mO}\Delta_d^h\bE(\dot{\bu}_0) \cdot \Delta_d^h (\bsig-\bsig_0) \phi^2  - \Delta_d^h(\bA  \dot{\bsig}_0+\bH \dot{\bxi}_0)\cdot (\Delta_d^h \bsig-\Delta_d^h \bsig_0)\phi^2 \ddd x\\
&\le \|(\Delta_d^h \bsig-\Delta_d^h \bsig_0)\phi\|_2 (\|\Delta_d^h\bE(\dot{\bu}_0)\|_2+\|\Delta_d^h(\bA  \dot{\bsig}_0+\bH \dot{\bxi}_0)\|_2)\\
&\le Ch(\|\nabla \dot{\bsig}_0\|_2+\|\nabla \dot{\bxi}_0\|_2)\|(\Delta_d^h \bsig-\Delta_d^h \bsig_0)\phi\|_2,
\end{split}
\end{equation}
where the last inequality follows from Sobolev characterization and  the assumptions on data.
Thus, we can now focus on the most critical term in \eqref{QQ1}, which is the first integral on the right hand side. To simplify the notation we use the abbreviation $x':=(x_1,\ldots, x_{d-1})$. Then using integration by parts, we observe
\begin{equation}
\label{QQ3}
\begin{split}
&\int_{\mO} \Delta_d^h\bE(\dot{\bu}-\dot{\bu}_0) \cdot \Delta_d^h (\bsig-\bsig_0) \phi^2\ddd x =\int_{\mO} \nabla \Delta_d^h(\dot{\bu}-\dot{\bu}_0) \cdot \Delta_d^h (\bsig-\bsig_0) \phi^2\ddd x \\
&=-2\int_{\mO} (\Delta_d^h(\dot{\bu}-\dot{\bu}_0)\otimes \nabla \phi) \cdot \Delta_d^h (\bsig-\bsig_0) \phi\ddd x \\
&\qquad + \sum_{i=1}^d\int_{\{x_d=0\}}\Delta_d^h(\dot{\bu}_i-\dot{\bu}_{0i})\Delta_d^h (\bsig_{id}-\bsig_{0id}) \phi^2\ddd x'\\
&\le C\|\Delta_d^h(\dot{\bu}-\dot{\bu}_0\|_2\|\Delta_d^h (\bsig-\bsig_0) \phi\|_2  \\
&\qquad + \sum_{i=1}^d\int_{\{x_d=0\}}\Delta_d^h(\dot{\bu}_i-\dot{\bu}_{0i})\Delta_d^h (\bsig_{id}-\bsig_{0id}) \phi^2\ddd x'\\
\end{split}
\end{equation}
The first term on the right hand side can be estimated by the use of characterization of Sobolev functions and the a~priori bound \eqref{infty2} as
\begin{equation}
\begin{split}\label{QQ4}
&\|\Delta_d^h(\dot{\bu}-\dot{\bu}_0\|_2\|\Delta_d^h (\bsig-\bsig_0) \phi\|_2\le Ch(\|\nabla \dot{\bu}\|_2 + \|\nabla \dot{\bu}_0\|_2)\|\Delta_d^h (\bsig-\bsig_0) \phi\|_2\\
&\quad \le Ch \|\Delta_d^h (\bsig-\bsig_0) \phi\|_2.
\end{split}
\end{equation}
For the second term on the right hand side, we first replace the differences by the corresponding integral, then use the fact that $\diver (\bsig-\bsig_0)=0$ (we also assume that $h\ll 1$ so that $\phi(x',0)=\phi(x',x_d)$ for arbitrary $x_d\in (0,h)$)
\begin{equation*}
\begin{split}%\label{QQ5}
&\int_{\{x_d=0\}}\Delta_d^h(\dot{\bu}_i-\dot{\bu}_{0i})\Delta_d^h (\bsig_{id}-\bsig_{0id}) \phi^2\ddd x'\\
&=\int_{\mathbb{R}^{d-1}}\left(\int_0^hD_d(\dot{\bu}_i-\dot{\bu}_{0i})\ddd x_d \right)\left(\int_0^hD_d (\bsig_{id}-\bsig_{0id}) \ddd x_d\right) \phi^2(x',0)\ddd x'\\
&\le \int_{\mathbb{R}^{d-1}}\left(\int_0^h |D_d\dot{\bu}\phi|+|\nabla (\dot{\bu}_{0}\phi)|\ddd x_d \right)\left|\sum_{j=1}^{d-1}\int_0^hD_j(\bsig_{ij}-\bsig_{0ij}) \ddd x_d\right| \phi(x',0)\ddd x'.
\end{split}
\end{equation*}
Then, we apply the H\"{o}lder inequality to conclude
\begin{equation}
\begin{split}\label{QQ5}
&\int_{\{x_d=0\}}\Delta_d^h(\dot{\bu}_i-\dot{\bu}_{0i})\Delta_d^h (\bsig_{id}-\bsig_{0id}) \phi^2\ddd x'\\
&\le Ch\sum_{j=1}^{d-1}\int_{\mathbb{R}^{d-1}}\left(\int_0^h (|D_d\dot{\bu}\phi|^2+|\nabla (\dot{\bu}_{0}\phi)|^2)\ddd x_d \right)^{\frac12}\times \\
&\qquad \times \left(\int_0^h (|D_j\bsig|^2+|\nabla \bsig_{0}|^2)\phi^2\ddd x_d\right)^{\frac12}\ddd x'\\
&\le Ch(\sum_{j=1}^{d-1}\|D_j \bsig \phi\|_2 + \|\nabla \bsig_0\|_2) \left(\int_{\mO \cap \{x_d\in (0,h)\}}\! \!\!\!\!\! (|D_d\dot{\bu}\phi|^2+|\nabla (\dot{\bu}_{0}\phi)|^2)\ddd x\right)^{\frac12}\\
&\le Ch^{\frac32}+ Ch\left(\int_{\mO \cap \{x_d\in (0,h)\}} |D_d\dot{\bu}\phi|^2\ddd x\right)^{\frac12}
\end{split}
\end{equation}
Hence, using \eqref{QQ2}--\eqref{QQ5} in \eqref{QQ1} and applying the Gronwall lemma, and using the fact that $\bsig(0)=\bsig_0$ and $\bxi(0)=\bxi_0$, and already proven a~priori estimates, we deduce \eqref{nor-start}. In case that the support of $\phi$ is located just closed to the Dirichlet boundary, we may use the Korn inequality (or the trace theorem) and to replace $|D_d\dot{\bu}\phi|$ by $|\bE(\dot{\bu})\phi|$ in the above estimate and to conclude \eqref{nor-start2}.

\end{proof}

\subsection{Estimate for $\nabla \dot{\bu}$ on a strip in terms of  $\dot{\bsig}$ and $\dot{\bxi}$}
In previous section, we deduce a uniform estimate for normal fractional derivatives in terms of $\nabla \dot{\bu}$, the right hand side of \eqref{nor-start}and \eqref{nor-start2}, respectively. In this part, we show how this term can be estimated in terms of $\dot{\bsig}$ and $\dot{\bxi}$. In fact, we prove two different estimates. The first one deals with the case that we have isotropic hardening and we are closed to the Neumann part of the boundary. The second case covers the kinematic hardening independently of Dirichlet or boundary data or the isotropic hardening in case of Dirichlet data.
\begin{Lemma}\label{Odhad-q}
For arbitrary $h\in (0,h_0)$ and $p>2$, the solution satisfies
\begin{equation}\label{odhad-q1}
\int_0^T \left(\int_{\mO \cap \{x_d\in (0,h)\}} \!\!\!\!\!\! |D_d\dot{\bu}\phi|^2\ddd x\right)^{\frac12}\ddd t\le C h^{\frac{p-2}{2p}}(1+ \int_0^T\|\dot{\bsig}\phi\|_{p}+\|\dot{\bxi}\phi\|_{p} \ddd t).
\end{equation}
In addition, if we consider  the  kinematic hardening, we have
\begin{equation}\label{odhad-q2}
\begin{split}
&\int_0^T \left(\int_{\mO \cap \{x_d\in (0,h)\}} \!\!\!\!\!\! |D_d\dot{\bu}\phi|^2\ddd x\right)^{\frac12}\ddd t\\
&\quad \le C(\delta)h^{\frac{p-2}{2p}}\left(1 + \sup_{s\in (0,4h_0)}\int_0^T\int_{\mO}\frac{|\phi \Delta_d^s \bsig|^2+|\phi \Delta_d^s \bxi|^2}{|s|^{\frac{3}{1-3\delta}(\delta+\frac{p-2}{p})}}\ddd x   \ddd t \right)^{\frac{1-3\delta}{6}}.
\end{split}
\end{equation}
Furthermore, for kinematic and isotropic hardening we also have
\begin{equation}\label{odhad-q3}
\begin{split}
&\int_0^T \left(\int_{\mO \cap \{x_d\in (0,h)\}} \!\!\!\!\!\! |\bE(\dot{\bu})\phi|^2\ddd x\right)^{\frac12}\ddd t\\
&\quad \le C(\delta)h^{\frac{p-2}{2p}}\left(1 + \sup_{s\in (0,4h_0)}\int_0^T\int_{\mO}\frac{|\phi \Delta_d^s \bsig|^2+|\phi \Delta_d^s \bxi|^2}{|s|^{\frac{3}{1-3\delta}(\delta+\frac{p-2}{p})}}\ddd x   \ddd t \right)^{\frac{1-3\delta}{6}}.
\end{split}
\end{equation}
\end{Lemma}
\begin{proof}
To prove the first case, we just use the H\"{o}lder and the Korn inequality as follows
\begin{equation}
\begin{split}
\int_{\mO \cap \{x_d\in (0,h)\}} \!\!\!\!\!\! |D_d\dot{\bu}\phi|^2\ddd x &\le C h^{\frac{p-2}{p}}\|D_d (\dot{\bu}\phi)\|^2_{L^p(\mO)}\le C h^{\frac{p-2}{p}}\|\bE(\dot{\bu}\phi)\|^2_{L^p(\mO)}\\
&\le C h^{\frac{p-2}{p}}(\|\bE(\dot{\bu})\phi\|^2_{L^p(\mO)}+1) \le C h^{\frac{p-2}{p}}(1+ \|\dot{\bsig}\phi\|^2_{p}+\|\dot{\bxi}\phi\|^2_{p}),
\end{split}
\end{equation}
where for the last inequality we used the equation \eqref{TPM}. Then we see that \eqref{odhad-q1} directly follows.

Next, we focus on \eqref{odhad-q2}. First, we recall the fractional Sobolev embedding $W^{\frac{p-2}{2p},2}(0,1)\hookrightarrow L^p(0,1)$ valid for all $p\in [2,\infty)$. Then with the help of the H\"{o}lder inequality, we obtain for almost all $x'$ and $t$ that
$$
\begin{aligned}
&\int_0^h |D_d\dot{\bu}\phi(t,x',x_d)|^2\ddd x_d\le h^{\frac{p-2}{p}}\left(\int_0^{2h_0} |D_d\dot{\bu}\phi(t,x',x_d)|^p\ddd x_d\right)^{\frac{2}{p}}\\
&\le Ch^{\frac{p-2}{p}}\left(\int_0^{2h_0} |D_d\dot{\bu}\phi(t,x',x_d)|^2\ddd x_d \right.)\\
&\qquad \left.+ \int_0^{2h_0}\int_0^{2h_0}\frac{|D_d\dot{\bu}\phi(t,x',x_d)-D_d\dot{\bu}\phi(t,x',y_d)|^2}{|x_d-y_d|^{1+\frac{p-2}{p}}}\ddd x_d \ddd y_d \right)
\end{aligned}
$$
Hence, using the Fubini theorem and the Korn inequality, we can continue with the estimate of the full integral as follows
$$
\begin{aligned}
&\int_0^T\int_{\mO \cap \{x_d\in (0,h)\}} \!\!\!\!\!\! |D_d\dot{\bu}\phi|^2\ddd x \ddd t\le Ch^{\frac{p-2}{p}}\left(\int_0^T\int_{(-1,1)^d}\int_0^{2h_0} |D_d\dot{\bu}\phi(t,x',x_d)|^2\ddd x\ddd t \right.\\
&\qquad \left.+ \int_0^T\int_{(-1,1)^d}\int_0^{2h_0}\int_0^{2h_0}\frac{|D_d\dot{\bu}\phi(t,x',x_d)-D_d\dot{\bu}\phi(t,x',y_d)|^2}{|x_d-y_d|^{1+\frac{p-2}{p}}}\ddd x_d \ddd y_d \ddd x' \ddd t\right)\\
&\le Ch^{\frac{p-2}{p}}\left(1 + \int_0^{4h_0}\int_0^T\int_{\mO}\frac{|D_d\Delta_d^s (\dot{\bu}\phi)|^2}{|s|^{1+\frac{p-2}{p}}}\ddd x   \ddd t \ddd s\right)\\
&\le Ch^{\frac{p-2}{p}}\left(1 + \int_0^{4h_0}\int_0^T\int_{\mO}\frac{|\Delta_d^s \bE(\dot{\bu}\phi)|^2}{|s|^{1+\frac{p-2}{p}}}\ddd x   \ddd t \ddd s\right)\\
&\le Ch^{\frac{p-2}{p}}\left(1 + \int_0^{4h_0}\int_0^T\int_{\mO}\frac{|\phi \Delta_d^s \bE(\dot{\bu})|^2}{|s|^{1+\frac{p-2}{p}}}\ddd x   \ddd t \ddd s\right)\\
&\le Ch^{\frac{p-2}{p}}\left(1 + \int_0^{4h_0}\int_0^T\int_{\mO}\frac{|\phi \Delta_d^s \dot{\bsig}|^2+|\phi \Delta_d^s \dot{\bxi}|^2}{|s|^{1+\frac{p-2}{p}}}\ddd x   \ddd t \ddd s\right),
\end{aligned}
$$
where we used the equation to evaluate $\bE(\dot{\bu})$ in terms of $\dot{\bsig}$ and $\dot{\bxi}$. Note that at this step we use the fact that we deal with the kinematic hardening. Now, we use Lemma~\ref{interpol1} to replace time derivative on the right hand side. Hence, doing so, we observe that for all $\delta>0$, we have
$$
\begin{aligned}
&\int_0^T\int_{\mO \cap \{x_d\in (0,h)\}} \!\!\!\!\!\! |D_d\dot{\bu}\phi|^2\ddd x \ddd t\\
&\quad \le C(\delta)h^{\frac{p-2}{p}}\left(1 + \int_0^{4h_0}\frac{1}{s^{1-\delta}}\left(\int_0^T\int_{\mO}\frac{|\phi \Delta_d^s \bsig|^2+|\phi \Delta_d^s \bxi|^2}{|s|^{\frac{3}{1-3\delta}(\delta+\frac{p-2}{p})}}\ddd x   \ddd t \right)^{\frac13 - \delta}\ddd s\right)\\
&\quad \le C(\delta)h^{\frac{p-2}{p}}\left(1 + \sup_{s\in (0,4h_0)}\int_0^T\int_{\mO}\frac{|\phi \Delta_d^s \bsig|^2+|\phi \Delta_d^s \bxi|^2}{|s|^{\frac{3}{1-3\delta}(\delta+\frac{p-2}{p})}}\ddd x   \ddd t \right)^{\frac13 - \delta}.
\end{aligned}
$$
Thus, using the H\"{o}lder inequality and the above estimate, we have
$$
\begin{aligned}
&\int_0^T\left(\int_{\mO \cap \{x_d\in (0,h)\}} \!\!\!\!\!\! |D_d\dot{\bu}\phi|^2\ddd x\right)^{\frac12} \ddd t\le C\left(\int_0^T\int_{\mO \cap \{x_d\in (0,h)\}} \!\!\!\!\!\! |D_d\dot{\bu}\phi|^2\ddd x \ddd t\right)^{\frac12}\\
&\quad \le C(\delta)h^{\frac{p-2}{2p}}\left(1 + \sup_{s\in (0,4h_0)}\int_0^T\int_{\mO}\frac{|\phi \Delta_d^s \bsig|^2+|\phi \Delta_d^s \bxi|^2}{|s|^{\frac{3}{1-3\delta}(\delta+\frac{p-2}{p})}}\ddd x   \ddd t \right)^{\frac16 - \frac{\delta}{2}},
\end{aligned}
$$
which is \eqref{odhad-q2}. To obtain \eqref{odhad-q3}, we proceed similarly, with the only change that from the beginning we have the point-wise estimate $|\bE(\dot {\bu})|\le C(|\dot{\bsig}|+ |\dot{\bxi}|)$, which follows from \eqref{TPM} and \eqref{TPMHe}, respectively.

%We use our description of the boundary and assume that $h_0$ and $\phi$ are  chosen such that $\phi(x',s)=0$ for all $s>1-h_0$. Then, we can rewrite the term on the left hand side of \eqref{odhad-q2} as follows.
%\begin{equation}
%\begin{split}
%&\int_{\mO \cap \{x_d\in (0,h)\}} \!\!\!\!\!\! |D_d\dot{\bu}\phi|^2\ddd x = -\int_{[-1,1]^{d-1}}\int_0^1 \Delta^h_d |D_d(\dot{\bu}\phi)|^2\ddd x\\
%&\quad\le \int_{[-1,1]^{d-1}}\int_0^1 |D_d(\dot{\bu}\phi)(x+he_d)+D_d(\dot{\bu}\phi)(x)|| D_d(\Delta^h_d(\dot{\bu}\phi))|\ddd x\\
%&\quad\le C\|\nabla (\dot{\bu}\phi)\|_2 \left(\int_{\mO} | D_d(\Delta^h_d(\dot{\bu}\phi))|^2\ddd x\right)^{\frac12}\le C \left(\int_{\mO} | \bE(\Delta^h_d(\dot{\bu}\phi))|^2\ddd x\right)^{\frac12}\\
%&\quad \le C \left(\int_{\mO} |\Delta^h_d (\bE(\dot{\bu}))|^2\phi^2 + h^2|\nabla \dot{\bu}|^2\ddd x\right)^{\frac12}\\
%&\quad \le C(h+ \|\Delta^h_d (\dot{\bsig})\phi\|_2 + \|\Delta^h_d (\dot{\bxi})\phi\|_2),
%\end{split}
%\end{equation}
%where we used the Korn inequality (since $\phi(x',1)=0$) and also the equation \eqref{TPM} and the fact that we deal kinematic hardening. Thus, \eqref{odhad-q2} easily follows. To finish the proof, we use \eqref{TPM}, and for \eqref{odhad-q3} can compute as above.
%\begin{equation}
%\begin{split}
%&\int_{\mO \cap \{x_d\in (0,h)\}} \!\!\!\!\!\! |\bE(\dot{\bu})\phi|^2\ddd x \le C\int_{\mO \cap \{x_d\in (0,h)\}} \!\!\!\!\!\! |\dot{\bsig}\phi|^2+|\dot{\bxi}\phi|^2\ddd x\\
%&\quad \le C(h+ \|\Delta^h_d (\dot{\bsig})\phi\|_2 + \|\Delta^h_d (\dot{\bxi})\phi\|_2),
%\end{split}
%\end{equation}
%and the proof of \eqref{odhad-q3} is complete.
\end{proof}

\subsection{Estimate for for $\dot{\bsig}$ and $\dot{\bxi}$ via anisotropic embedding}

\begin{Lemma}
Let $p\in (2,\frac{2(d-1)}{d-2})$ and $\beta>0$ be given as
\begin{equation}\label{beta}
\beta:=\frac{p-2}{4(d-1) - 2p(d-2)}
\end{equation}
Then for any solution and any $\delta>0$ there holds
\begin{equation}\label{embde}
\int_0^T\|\dot{\bsig}\phi\|_{p}+\|\dot{\bxi}\phi\|_{p} \ddd t \le \frac{C}{\delta} \left(1+\sup_{h\in (0,1)}\int_0^T \int_{\mO} \frac{|\Delta^h_d\dot{\bsig}\phi|^2+|\Delta^h_d\dot{\bxi}\phi|^2}{h^{2(\beta+\delta)}}  \ddd x \ddd t\right)^{\frac12}.
\end{equation}
\end{Lemma}
\begin{proof}
We use the following version of anisotropic embedding. We assume that $f\in L^2(0,1; L^2((-1,1)^d))$ and denote $B:=(-1,1)^{d-1}$.
We define $\lambda\in (0,1)$ by the relation
$$
\lambda := \frac{1}{2\beta(d-1)+1}
$$
and using the Sobolev embedding, we have (recall the definition of $\beta$ in \eqref{beta})
\begin{equation}\label{needed}
\begin{split}
W^{\lambda\beta,2}(0,1)\hookrightarrow L^{p} (0,1),\\
W^{\frac{1-\lambda}{2},2}(B)\hookrightarrow L^{p} (B).
\end{split}
\end{equation}
Then, we can use the cross-interpolation in Sobolev-Bochner spaces and the above embedding to observe
$$
\|f\|^2_p \le C \|f\|^2_{W^{\lambda\beta,2}(0,1; W^{\frac{1-\lambda}{2},2}(B))}\le C(\|f\|^2_{L^2(0,1; W^{\frac{\alpha}{2},2}(B))}+\|f\|^2_{W^{\frac{\alpha\lambda\beta}{\alpha-1+\lambda},2}(0,1; L^2(B))})
$$
for arbitrary $1-\lambda<\alpha <1$. Next, we use the above inequality for $\dot{\bsig}\phi$ and integrate also over time $t\in (0,T)$. Thus, using also the definition of fractional Sobolev norm, we have (using also the properties of $\phi$ and the a~priori estimate \eqref{infty1})
$$
\begin{aligned}
\int_0^T & \|\dot{\bsig}\phi\|_p^2 \le C\int_0^T \|\dot{\bsig}\phi\|^2_{L^2(0,1; W^{\frac{\alpha}{2},2}(B))}+\|\dot{\bsig}\phi\|^2_{W^{\frac{\alpha\lambda\beta}{\alpha-1+\lambda},2}(0,1; L^2(B))}\ddd t\\
&\le C + C\int_0^T \int_0^1 \int_B \int_B  \frac{|\dot{\bsig}\phi(t,x',x_d)-\dot{\bsig}\phi(t,y',x_d)|^2}{|x'-y'|^{d-1+\alpha}}\ddd x' \ddd y' \ddd x_d  \ddd t\\
&\quad + C\int_0^T \int_0^1 \int_0^1 \int_B \frac{|\dot{\bsig}\phi(t,x',x_d)-\dot{\bsig}\phi(t,x',y_d)|^2}{|x_d-y_d|^{1+\frac{2\alpha\lambda\beta}{\alpha-1+\lambda}}} \ddd x'  \ddd x_d \ddd y_d \ddd t\\
&\le C + C\int_0^T \int_{\mO} \int_{2B}  \frac{|\dot{\bsig}\phi(t,x'+z',x_d)-\dot{\bsig}\phi(t,x',x_d)|^2}{|z'|^{d-1+\alpha}}\ddd z' \ddd x  \ddd t\\
&\quad + C\int_0^T \int_{\mO}\int_0^1 \frac{|\dot{\bsig}\phi(t,x',x_d)-\dot{\bsig}\phi(t,x',x_d+h)|^2}{h^{1+\frac{2\alpha\lambda\beta}{\alpha-1+\lambda}}}  \ddd h \ddd x \ddd t\\
&\le C + C\sup_{i=1,\ldots, d-1} \sup_{h\in (0,1)}\int_0^T \int_{\mO} \frac{|\Delta^h_i \dot{\bsig}\phi|^2}{h} \ddd x  \ddd t  \int_{2B}  \frac{1}{|z'|^{d-2+\alpha}}\ddd z'\\
&\quad +C\sup_{h\in (0,1)}\int_0^T \int_{\mO} \frac{|\Delta^h_d\dot{\bsig}\phi|^2}{h^{2(\beta+\delta)}}  \ddd x \ddd t \int_0^1 \frac{1}{s^{1-2(\beta+\delta)+\frac{2\alpha\lambda\beta}{\alpha-1+\lambda}}}  \ddd s \\
&\le \frac{C}{1-\alpha}  + \frac{C}{2(\beta+\delta)-\frac{2\alpha\lambda\beta}{\alpha-1+\lambda}}\sup_{h\in (0,1)}\int_0^T \int_{\mO} \frac{|\Delta^h_d\dot{\bsig}\phi|^2}{h^{2(\beta+\delta)}}  \ddd x \ddd t.
\end{aligned}
$$
Consequently, for any $\delta>0$, we can find $\alpha\in (0,1)$ such that
$$
2(\beta+\delta)>\frac{2\alpha\lambda\beta}{\alpha-1+\lambda}
$$
and \eqref{embde} for $\dot{\bsig}$ follows by H\"{o}lder inequality. The same scheme is used for the estimate for $\dot{\bxi}$.

\end{proof}

\subsection{Final estimate for normal derivatives - the case of Dirichlet boundary or the case of kinematic hardening}
In this part we finish the proof of the main theorem. In particular, we focus on the normal derivative estimates stated in \eqref{ident-limit} and \eqref{ident-limitHe}. We start with \eqref{nor-start} to conclude the starting estimate
\begin{equation}
\label{Mina}
\sup_{t\in (0,T)}\sup_{h\in (0,1)} \frac{\|\phi \Delta_d^h \bsig(t)\|_2^2 + \|\phi \Delta_d^h \bxi(t)\|_2^2}{h}\le C.
\end{equation}
Then, we use  \eqref{nor-start} (in case of kinematic hardening) or in \eqref{nor-start2} (in case of Dirichlet boundary condition), and the term on the right hand side is replaced by the corresponding estimates in \eqref{odhad-q2} and \eqref{odhad-q3}, respectively. Thus we conclude in both cases that
\begin{equation*}
\begin{split}
\sup_{t\in (0,T)} &\|\phi \Delta_d^h \bsig(t)\|_2^2 + \|\phi \Delta_d^h \bxi(t)\|_2^2 \\
&\le Ch^{\frac32}+ C(\delta)h^{1+\frac{p-2}{2p}}\left(1 + \sup_{s\in (0,4h_0)}\int_0^T\int_{\mO}\frac{|\phi \Delta_d^s \bsig|^2+|\phi \Delta_d^s \bxi|^2}{|s|^{\frac{3}{1-3\delta}(\delta+\frac{p-2}{p})}}\ddd x   \ddd t \right)^{\frac{1-3\delta}{6}}.
\end{split}
\end{equation*}
Hence, it follows that
\begin{equation}\label{nor-start-aaa}
\begin{split}
\sup_{t\in (0,T)}\sup_{h\in (0,4h_0)} &\int_{\mO}\frac{|\phi \Delta_d^h \bsig(t)|^2 + |\phi \Delta_d^h \bxi(t)|^2}{h^{1+\frac{p-2}{2p}}}\ddd x \\
&\le C(\delta)\left(1 + \sup_{s\in (0,4h_0)}\int_0^T\int_{\mO}\frac{|\phi \Delta_d^s \bsig|^2+|\phi \Delta_d^s \bxi|^2}{|s|^{\frac{3}{1-3\delta}(\delta+\frac{p-2}{p})}}\ddd x   \ddd t \right)^{\frac{1-3\delta}{6}},
\end{split}
\end{equation}
provided that the right hand side is finite. Hence, we can start the iteration with \eqref{Mina}. In the first step, thanks to \eqref{Mina}, we can set in \eqref{nor-start-aaa} arbitrary $p<3$ and find sufficiently small $\delta>0$ such that
$$
\frac{3}{1-3\delta}\left(\delta+\frac{p-2}{p}\right)<1
$$
and we immediately get an improvement of \eqref{Mina}. Consequently, iterating such procedure is possible as long as
$$
1+\frac{p-2}{2p}>\frac{3(p-2)}{p} \qquad \Leftrightarrow \qquad  \frac15>\frac{p-2}{2p}.
$$
Thus, it follows that we are able to obtain
\begin{equation}\label{hola}
\begin{split}
\sup_{t\in (0,T)}\sup_{h\in (0,4h_0)} &\int_{\mO}\frac{|\phi \Delta_d^h \bsig(t)|^2 + |\phi \Delta_d^h \bxi(t)|^2}{h^{\frac65 -\delta}}\ddd x \le C(\delta)
\end{split}
\end{equation}
and by \eqref{interpol2} we also have
\begin{equation}\label{hola12}
\begin{split}
\sup_{h\in (0,4h_0)} &\int_0^T\int_{\mO}\frac{|\phi \Delta_d^h \dot{\bsig}|^2 + |\phi \Delta_d^h \dot{\bxi}|^2}{h^{\frac25 -\delta}}\ddd x \ddd t\le C(\delta),
\end{split}
\end{equation}
which finishes the proof for kinematic hardening or the case of Dirichlet boundary conditions.

%where we used the H\"{o}lder inequality for the last estimate. Finally, we use the interpolation result \eqref{interpol2} and obtain from the above inequality that
%\begin{equation}\label{nor-start-bbb}
%\begin{split}
%\sup_{t\in (0,T)} &\|\phi \Delta_d^h \bsig(t)\|_2^2 + \|\phi \Delta_d^h \bxi(t)\|_2^2 \\
%&\le Ch^{\frac32}+ C(\delta)h\left(\int_0^T\|\phi\Delta^h_d \dot{\bsig}\|^{2}_2 + \|\phi\Delta^h_d \dot{\bxi}\|^{2}_2\ddd t\right)^{\frac{1}{12}-\delta},
%\end{split}
%\end{equation}
%valid for all $\delta>0$. Hence, applying the Young inequality, we see that
%\begin{equation}\label{nor-start-bbb-f}
%\begin{split}
%\sup_{t\in (0,T)} &\|\phi \Delta_d^h \bsig(t)\|_2^2 + \|\phi \Delta_d^h \bxi(t)\|_2^2 \le Ch^{\frac32}+ C(\delta)h^{\frac{12}{11}-\delta}\le C(\delta)h^{\frac{12}{11}-\delta},
%\end{split}
%\end{equation}
%valid for all $\delta>0$. This finishes the proof of the first estimate in \eqref{ident-limit} and \eqref{ident-limitHe}. Using again \eqref{interpol2}, we finally deduce
%\begin{equation}\label{nor-start-bbb-g}
%\begin{split}
%\int_0^T\|\phi \Delta_d^h \dot{\bsig}\|_2^2 + \|\phi \Delta_d^h \dot{\bxi}\|_2^2 \ddd t \le  C(\delta)h^{\frac{4}{11}-\delta}
%\end{split}
%\end{equation}
%for arbitrary $\delta>0$. Hence, the second part of estimate in \eqref{ident-limit} and \eqref{ident-limitHe} follows.

\subsection{Final estimate for normal derivatives - the case of Neumann boundary and the  isotropic hardening}
In this case we use the anisotropic embedding. We again start with \eqref{nor-start} to conclude the starting estimate
\begin{equation}
\label{Minabb}
\sup_{t\in (0,T)}\sup_{h\in (0,1)} \frac{\|\phi \Delta_d^h \bsig(t)\|_2^2 + \|\phi \Delta_d^h \bxi(t)\|_2^2}{h}\le C.
\end{equation}
Then we start with iteration. Using \eqref{nor-start} and \eqref{odhad-q1}, we observe that
\begin{equation}\label{nor-starttr}
\begin{split}
\sup_{t\in (0,T)}\sup_{h\in (0,1)} &\frac{\|\phi \Delta_d^h \bsig(t)\|_2^2 + \|\phi \Delta_d^h \bxi(t)\|_2^2}{h^{1+\frac{p-2}{2p}}} \le C (1+ \int_0^T\|\dot{\bsig}\phi\|_{p}+\|\dot{\bxi}\phi\|_{p} \ddd t).
\end{split}
\end{equation}
Next, recalling the definition of $\beta$, see \eqref{beta},
$$
\beta:=\frac{p-2}{4(d-1) - 2p(d-2)},
$$
and combining \eqref{embde} and \eqref{interpol2} and the Young inequality, we also have
\begin{equation}\label{embde-gh}
\int_0^T\|\dot{\bsig}\phi\|_{p}+\|\dot{\bxi}\phi\|_{p} \ddd t \le \frac{C}{\delta} \left(1+\sup_{h\in (0,1)}\int_0^T \int_{\mO} \frac{|\Delta^h_d{\bsig}\phi|^2+|\Delta^h_d{\bxi}\phi|^2}{h^{6(\beta+2\delta)}}  \ddd x \ddd t\right),
\end{equation}
where $\delta>0$ is arbitrary. Thus, combining \eqref{nor-starttr} and \eqref{embde-gh}, we have
\begin{equation}\label{nor-jupi}
\begin{split}
\sup_{t\in (0,T)}\sup_{h\in (0,1)} &\frac{\|\phi \Delta_d^h \bsig(t)\|_2^2 + \|\phi \Delta_d^h \bxi(t)\|_2^2}{h^{1+\frac{p-2}{2p}}} \\
&\le  \frac{C}{\delta} \left(1+\sup_{h\in (0,1)}\int_0^T \int_{\mO} \frac{|\Delta^h_d{\bsig}\phi|^2+|\Delta^h_d{\bxi}\phi|^2}{h^{6(\beta+2\delta)}}  \ddd x \ddd t\right).
\end{split}
\end{equation}
Thus, we can again start with iteration, which is possible as long as
$$
1+\frac{p-2}{2p}> 6\beta \quad \Leftrightarrow \quad  \beta<\frac{2d-7+\sqrt{1+4d^2+20d}}{24(d-1)}.
$$
Consequently, we have the final estimate (here $\delta\in (0,1)$ is arbitrary)
\begin{equation*}
\begin{split}
\sup_{t\in (0,T)}\sup_{h\in (0,1)} &\frac{\|\phi \Delta_d^h \bsig(t)\|_2^2 + \|\phi \Delta_d^h \bxi(t)\|_2^2}{h^{\frac{2d-7+\sqrt{1+4d^2+20d}}{4(d-1)}-\delta}} \le  \frac{C}{\delta}.
\end{split}
\end{equation*}
and thanks to time interpolation \eqref{interpol2}, it leads to
\begin{equation*}
\begin{split}
\sup_{h\in (0,1)} &\int_0^T\frac{\|\phi \Delta_d^h \dot{\bsig}\|_2^2 + \|\phi \Delta_d^h \dot{\bxi}\|_2^2}{h^{\frac{2d-7+\sqrt{1+4d^2+20d}}{12(d-1)}-\delta}}\ddd t \le  \frac{C}{\delta},
\end{split}
\end{equation*}
where the constant $C(\delta)$ explodes as $\delta \to 0_+$. Hence, the proof is complete.

%\bibliography{mbul8060}
%\bibliographystyle{plain}

\end{document}